\documentclass{article}[10pt,a4paper]
\usepackage{latexsym}
\usepackage{amssymb}
\usepackage{amsmath,amsthm}
\usepackage{mathrsfs}
\usepackage{graphicx}
\usepackage{verbatim}

\usepackage[all]{xy}

\newtheorem{theo}{Theorem}[section]

\newtheorem{lemma}[theo]{Lemma}

\newtheorem{corollary}[theo]{Corollary}
\newtheorem{prop}[theo]{Proposition}

\newcommand {\ZZ} {\mathbb {Z}}

\newcommand {\CC} {\mathbb {C}}

\newcommand {\QQ} {\mathbb {Q}}
\newcommand {\NN} {\mathbb {N}}

\newcommand{\Cal}{\mathcal}

\newcommand{\nsimeq}{\hbox{$\simeq${\raise0.3ex\hbox{\kern -0.55em ${\scriptscriptstyle /}$}}\ }}

\renewcommand{\ss}{\mathfrak{s}}
\newcommand{\kk}{\mathfrak{k}}

\newcommand{\hh}{\mathfrak{h}}

\renewcommand{\gg}{\mathfrak{g}}

\renewcommand {\phi} {\varphi}

\newcommand{\refth}[1]{Theorem \ref{#1}}
\newcommand{\refle}[1]{Lemma \ref{#1}}

\newcommand{\refcor}[1]{Corollary \ref{#1}}
\newcommand{\refprop}[1]{Proposition \ref{#1}}

\newcommand{\refeq}[1]{(\ref{#1})}

\newcommand{\rk}{\mathrm{rk}}

\renewcommand{\sp}{\mathrm{sp}}
\newcommand{\gl}{\mathrm{gl}}
\renewcommand{\sl}{\mathrm{sl}}
\newcommand{\so}{\mathrm{so}}
\renewcommand{\o}{\mathrm{o}}

\newcommand {\quotst}{\div}

\def\cplus{\hbox{$\subset${\raise0.3ex\hbox{\kern -0.55em ${\scriptscriptstyle +}$}}\ }}

\def\clplus{\hbox{$\subset${\raise0.3ex\hbox{\kern -0.55em ${\scriptscriptstyle +}$}}\ }}
\def\crplus{\hbox{$\supset${\raise1.05pt\hbox{\kern -0.55em ${\scriptscriptstyle +}$}}\ }}

\author{Sergei Markouski}
\title{
Locally simple subalgebras of diagonal Lie algebras\\
}\date{}
\begin{document}
\maketitle \abstract{ We describe, up to isomorphism, all locally
simple subalgebras of any diagonal locally simple Lie algebra.

Key words (2010 MSC): 17B05 and 17B65.}

\section{Introduction}
A Lie algebra $\gg$ is \emph{locally finite} if any finite subset
$S$ of $\gg$ is contained in a finite-dimensional Lie subalgebra
$\gg(S)$ of $\gg$. If, for any $S$, $\gg(S)$ can be chosen simple
(semisimple), $\gg$ is called \emph{locally simple (semisimple)}.
In 1998, A. Baranov introduced the class of diagonal locally
finite Lie algebras and established their general properties, see
\cite{B1}, \cite{B2}. Moreover, an explicit description of the
more special class of diagonal locally simple Lie algebras was
obtained by A. Baranov and A. Zhilinskii in \cite{BZ}, where they
classified diagonal direct limits of simple complex Lie algebras
up to isomorphism. In the present paper we work with the latter
class of Lie algebras, and throughout the paper a diagonal Lie
algebra will be assumed locally simple. Particular examples of
such algebras are the classical infinite-dimensional complex Lie
algebras $\sl(\infty)$, $\so(\infty)$, and $\sp(\infty)$, which
can be defined as the unions $\cup_{i\in\ZZ_{> 1}}\sl(i)$,
$\cup_{i\in\ZZ_{> 1}}\o(i)$, and $\cup_{i\in\ZZ_{> 1}}\sp(2i)$,
respectively, for any inclusions $\sl(i)\subset\sl(i+1)$,
$\o(i)\subset\o(i+1)$, and $\sp(2i)\subset\sp(2i+2)$, $i>1$.
Moreover, the latter Lie algebras are the only
countable-dimensional finitary locally simple complex Lie
algebras, see \cite{B3}, \cite{B4}, \cite{BS}.

The semisimple subalgebras of semisimple finite-dimensional
complex Lie algebras were described by A. Malcev and E. Dynkin
more than half a century ago \cite{M}, \cite{D}. Recently, I.
Dimitrov and I. Penkov characterized all locally semisimple
subalgebras of $\sl(\infty)$, $\so(\infty)$, and $\sp(\infty)$
\cite{DP}. The same problem is of interest for the more general
class of diagonal Lie algebras. It makes sense to first restrict
the problem to describing, up to isomorphism, all locally simple
subalgebras of diagonal Lie algebras. The purpose of this paper is
to present a solution of the latter problem.

\section{Preliminaries}
The base field is $\CC$. We assume that all Lie algebras
considered are finite dimensional or countable dimensional. When
considering classical simple Lie algebras, we consider the three
types $A$, $C$, and $O$, where $O$ stands for both types $B$ and
$D$.

A classical simple Lie subalgebra $\gg_1$ of a finite-dimensional
classical simple Lie algebra $\gg_2$ is called \emph{diagonal} if
there is an isomorphism of $\gg_1$-modules
$$ \displaystyle V_2\downarrow\gg_1\cong\underbrace{V_1\oplus\ldots\oplus
V_1}_l\oplus\underbrace{V_1^{\ast}\oplus\ldots\oplus
V_1^{\ast}}_r\oplus\underbrace{T_1\oplus\ldots\oplus T_1}_z,
$$ where $V_i$ is the natural $\gg_i$-module
($i=1,2$), $V_1^{\ast}$ is the dual of $V_1$, and $T_1$ is the
one-dimensional trivial $\gg_1$-module. The triple $(l,r,z)$ is
called the \emph{signature} of $\gg_1$ in $\gg_2$. An injective
homomorphism $\varepsilon:\gg_1\rightarrow\gg_2$ is
\emph{diagonal} if $\varepsilon(\gg_1)$ is a diagonal subalgebra
of $\gg_2$. The \emph{signature} of $\varepsilon$ is by definition
the signature of $\varepsilon(\gg_1)$ in $\gg_2$.

An \emph{exhaustion} $$\gg_1\subset\gg_2\subset\cdots$$ of a
locally finite Lie algebra $\gg$ is a direct system of
finite-dimensional Lie subalgebras of $\gg$ such that the direct
limit Lie algebra $\underrightarrow{\lim}\,\gg_n$ is isomorphic to
$\gg$. A locally simple Lie algebra $\ss$ is \emph{diagonal} if it
admits an exhaustion by simple subalgebras $\ss_i$ such that all
inclusions $\ss_i\subset\ss_{i+1}$ are diagonal.

The following result is due to A. Baranov.
\begin{prop}
Any locally simple subalgebra of a diagonal Lie algebra is
diagonal.
\end{prop}
\begin{proof}
Let $\ss$ be a locally simple subalgebra of a diagonal Lie algebra
$\ss'$. Corollary 5.11 in \cite{B1} claims that a locally simple
Lie algebra is diagonal if and only if it admits an injective
homomorphism into a Lie algebra associated with some locally
finite associative algebra. Hence $\ss'$ admits an injective
homomorphism into a Lie algebra $\gg$ associated with some locally
finite associative algebra. Then there is an injective
homomorphism $\ss\rightarrow\ss'\rightarrow\gg$, so $\ss$ is
diagonal.
\end{proof}

This result reduces the study of locally simple subalgebras of
diagonal Lie algebras to the study of diagonal subalgebras.

Next we introduce the notion of index of a simple subalgebra in a
simple Lie algebra. This notion goes back to E. Dynkin \cite{D}.
For a simple finite-dimensional Lie algebra $\gg$ we denote by
$\langle\;,\;\rangle_\gg$ the invariant non-degenerate symmetric
bilinear form on $\gg$ normalized so that
$\langle\alpha,\alpha^\vee\rangle_\gg =2$ for any long root
$\alpha$ of $\gg$. If $\phi:\ss\to\gg$ is an injective
homomorphism of simple Lie algebras, then $\langle
x,y\rangle_\phi:=\langle \phi(x),\phi(y)\rangle_\gg$ is an
invariant non-degenerate symmetric bilinear form on $\ss$.
Consequently,
\[\langle x,y\rangle_\phi=I_\ss^\gg(\phi)\langle x, y\rangle_\ss\]
for some scalar $I_\ss^\gg(\phi)$. By definition $I_\ss^\gg(\phi)$
is the \emph{index} of $\ss$ in $\gg$. If $\phi$ is clear from the
context, we will simply write $I_\ss^\gg$. If $U$ is any
finite-dimensional $\ss$-module, then the \emph{index}
$I_{\ss}(U)$ of $U$ is defined as $I_\ss^{\sl(U)}$, where $\ss$ is
mapped into $\sl(U)$ through the module $U$. The following
properties of the index are established in \cite{D}.

\begin{prop}\label{Dprop}

\begin{itemize}
\item[(i)] $I_\ss^\gg\in\ZZ_{\geq 0}$. \item[(ii)] $I_\ss^\kk
I_\kk^\gg=I_\ss^\gg$. \item[(iii)] $I_{\ss}(U_1\oplus\cdots\oplus
U_n)=I_{\ss}(U_1)+\cdots+I_{\ss}(U_n)$.\item[(iv)] If $U$ is an
$\ss$-module with highest weight $\lambda$ (with respect to some
Borel subalgebra), then $I_\ss(U)=\frac{\dim U}{\dim
\ss}\langle\lambda,\lambda+2\rho\rangle_\ss$, where $2\rho$ is the
sum of all the positive roots of $\ss$.
\end{itemize}
\end{prop}
\begin{corollary}\label{Dcor}
Let $\ss$ and $\gg$ be finite-dimensional classical simple Lie
algebras of the same type ($A$, $C$, or $O$). If $\ss$ is a
diagonal subalgebra of $\gg$ of signature $(l,r,z)$, then
$I_\ss^\gg(\varepsilon)=l+r$.
\end{corollary}
\begin{proof}
Indeed, if $V$ is the natural $\ss$-module then clearly
$I_\ss(V)=I_\ss(V^{\ast})$, and (iii) implies the result for type
$A$ algebras. If $\ss$ and $\gg$ are of type $O$ or $C$ then the
result follows from the observation in \cite{DP} that
$I_\ss^{\sp(U)}=I_\ss(U)$ and $I_\ss^{\so(U)}=\frac{1}{2}I_\ss(U)$
when $U$ admits a corresponding invariant form. This latter
observation is also a corollary from \cite{D}.
\end{proof}

Let us now recall several notions introduced by Baranov and
Zhilinskii, and state the main result of \cite{BZ}, namely the
classification of diagonal Lie algebras.

Let $p_1=2,p_2=3,\dots$ be the increasing sequence of all prime
numbers. A map from the set $\{p_1,p_2,\ldots\}$ into the set
$\{0,1,2,\ldots\}\bigcup\{\infty\}$ is called a \emph{Steinitz
number}. The Steinitz number which has value $\alpha_1$ at $p_1$,
$\alpha_2$ at $p_2$, etc$.$ will be denoted by
$p_1^{\alpha_1}p_2^{\alpha_2}\cdots$. Let
$\Pi=p_1^{\alpha_1}p_2^{\alpha_2}\cdots$ and
$\Pi'=p_1^{\alpha'_1}p_2^{\alpha'_2}\cdots$ be two Steinitz
numbers. We put
$\Pi\,\Pi'=p_1^{\alpha_1+\alpha'_1}p_2^{\alpha_2+\alpha'_2}\cdots$,
and we say that $\Pi$ \emph{divides} $\Pi'$ (or $\Pi|\Pi'$) if and
only if $\alpha_1\leq\alpha'_1,\,\alpha_2\leq\alpha'_2,\,\dots$.
In the latter case we write
$\quotst(\Pi',\Pi)=p_1^{\alpha'_1-\alpha_1}p_2^{\alpha'_2-\alpha_2}\cdots$,
where by convention $p_i^{\infty-\infty}=1$ for any $i$. We also
define the greatest common divisor $\mathrm{GCD}(\Pi,\Pi')$ as
$p_1^{\min(\alpha_1,\alpha'_1)}p_2^{\min(\alpha_2,\alpha'_2)}\cdots$.

Let $q\in\QQ$. We write $\Pi=q\Pi'$ (or $q\in\frac{\Pi}{\Pi'}$) if
there exists $n\in\NN$ such that $nq\in\NN$ and $n\Pi=nq\Pi'$. If
there exists $0\neq q\in\QQ$ such that $\Pi=q\Pi'$, then we say
that $\Pi$ and $\Pi'$ are $\QQ$-\emph{equivalent} and denote this
relation by $\displaystyle \Pi\stackrel{\QQ}{\sim}\Pi'$. Suppose
$q\in\frac{\Pi}{\Pi'}$ for some $0\neq q\in\QQ$. If $p^\infty$
divides $\Pi$, then $p^\infty$ also divides $\Pi'$ and so
$\Pi=qp^k\Pi'$ for all $k\in\ZZ$. Hence in this case
$\{qp^k\}_{k\in\ZZ}$ is a subset of $\frac{\Pi}{\Pi'}$ in our
notation. On the other hand, if there is no prime $p$ with
$p^\infty$ dividing $\Pi$, then the set $\frac{\Pi}{\Pi'}$
consists of the only element $q$. If $\Cal S = (s_1,s_2,\ldots)$
is a sequence of positive integers, $\mathrm{Stz}(\Cal S)$ denotes
the infinite product $\displaystyle \prod_{i=1}^{\infty}s_i$
considered as a Steinitz number.

Let $\ss$ be an infinite-dimensional diagonal Lie algebra, so
there is an exhaustion $\ss=\cup_i\ss_i$ with all inclusions
$\ss_i\subset\ss_{i+1}$ being diagonal. Without loss of generality
we may assume that all $\ss_i$ are of the same type $X$ ($X=A$,
$C$, or $O$), and we say that $\ss$ is of type $X$. Note that a
diagonal Lie algebra can be of more than one type. The triple
$(l_i,r_i,z_i)$ denotes the signature of the homomorphism
$\ss_i\rightarrow\ss_{i+1}$ and $n_i$ denotes the dimension of the
natural $\ss_i$-module. We assume that $r_i=0$ if $X$ is not $A$
(for all classical Lie algebras of type other than $A$ the natural
representation is isomorphic to its dual). We also assume that
$l_i\geq r_i$ for all $i$ for type $A$ algebras. (This does not
restrict generality as one can apply outer automorphisms to a
suitable subexhaustion if necessary.) Finally, if not stated
otherwise, we assume that $n_1=1,\,l_1=n_2,\,r_1=z_1=0.$ Denote by
$\Cal T$ the sequence of all such triples
$\{(l_i,r_i,z_i)\}_{i\in\NN}$. We will write $\ss=X(\Cal T)$ which
make sense up to isomorphism.

Set $s_i=l_i+r_i$, $c_i=l_i-r_i$ ($i\geq 1$), $\Cal S =
(s_i)_{i\in\NN}$, $\Cal C=(c_i)_{i\in\NN}$. Put
$\delta_i=\frac{s_1\cdots s_{n-1}}{n_i}$. Then
$\delta_{i+1}=\frac{s_1\cdots s_n}{n_{i+1}}=\frac{s_1\cdots
s_{n-1}}{n_i+(z_i/s_i)}\leq\delta_i$. The limit
$\displaystyle\delta=\lim_{i\rightarrow\infty}\delta_i$ is called
the \emph{density index} of $\Cal T$ and is denoted by
$\delta(\Cal T)$. Since $\delta_2=s_1/n_2=1$, we have
$0\leq\delta\leq 1$. If $\delta =0$ then the sequence of triples
$\Cal T$ is called \emph{sparse}. If there exists $i$ such that
$\delta_j=\delta_i\neq 0$ for all $j>i$, the sequence is called
\emph{pure}. We say that $\Cal T$ is \emph{dense} if $0<\delta
<\delta_i$ for all $i$.

If there exists $i$ such that $c_j=s_j$ for all $j\geq i$, then
$\Cal T$ is called \emph{one-sided} (in which case we can and will
assume that $c_j=s_j$ for all $j\geq 1$). Otherwise it is called
\emph{two-sided}. If, for each $i$, there exists $j>i$ such that
$c_j=0$, then $\Cal T$ is called \emph{symmetric}. Otherwise it is
called \emph{non-symmetric}. In the latter case we will assume
that $c_i>0$ for all $i\geq 1$. Set $\sigma_i=\frac{c_1\cdots
c_i}{s_1\cdots s_i}$. The limit
$\displaystyle\sigma=\lim_{i\rightarrow\infty}\sigma_i$ is called
the \emph{symmetry index} of $\Cal T$ and is denoted by
$\sigma(\Cal T)$. Observe that $0\leq\sigma\leq 1$. Two-sided
non-symmetric sequences $\Cal T$ with $\sigma(\Cal T)=0$ are
called \emph{weakly non-symmetric}, and those with $\sigma(\Cal
T)\neq 0$ are called \emph{strongly non-symmetric}.

The classification of the infinite-dimensional diagonal locally
simple Lie algebras is given by the following two theorems.

\begin{theo}\textbf{\cite{BZ}}\label{BZh1}
Let $X=A$, $C$, or $O$. Let $\Cal T=\{(l_i,r_i,z_i)\}$ and $\Cal
T'=\{(l'_i,r'_i,z'_i)\}$, where $r_i=r'_i=0$ if $X\neq A$. Set
$\delta=\delta(\Cal T)$, $\sigma=\sigma(\Cal T)$,
$\delta'=\delta(\Cal T')$, $\sigma'=\sigma(\Cal T')$. Then $X(\Cal
T)\cong X(\Cal T')$ if and only if the following conditions hold.

\begin{itemize}
\item[($\Cal A_1$)] The sequences $\Cal T$ and $\Cal T'$ have the
same density type.

\item[($\Cal A_2$)] $\mathrm{Stz}(\Cal
S)\stackrel{\QQ}{\sim}\mathrm{Stz}(\Cal S')$.

\item[($\Cal A_3$)]
$\frac{\delta}{\delta'}\in\frac{\mathrm{Stz}(\Cal
S)}{\mathrm{Stz}(\Cal S')}$ for dense and pure sequences.

\item[($\Cal B_1$)] The sequences $\Cal T$ and $\Cal T'$ have the
same symmetry type.

\item[($\Cal B_2$)] $\mathrm{Stz}(\Cal
C)\stackrel{\QQ}{\sim}\mathrm{Stz}(\Cal C')$ for two-sided
non-symmetric sequences.

\item[($\Cal B_3$)] There exists $\alpha\in\frac{\mathrm{Stz}(\Cal
S)}{\mathrm{Stz}(\Cal S')}$ such that
$\alpha\frac{\sigma}{\sigma'}\in\frac{\mathrm{Stz}(\Cal
C)}{\mathrm{Stz}(\Cal C')}$ for two-sided strongly non-symmetric
sequences. Moreover, $\alpha=\frac{\delta}{\delta'}$ if in
addition the triple sequences are dense or pure.
\end{itemize}

\end{theo}

\begin{theo}\textbf{\cite{BZ}}\label{BZh2}
Let $\Cal T=\{(l_i,r_i,z_i)\}$, $\Cal T'=\{(l'_i,0,z'_i)\}$, and
$\Cal T''=\{(l''_i,0,z''_i)\}$.

\begin{itemize}
\item[(i)] $A(\Cal T)\cong O(\Cal T')$ (resp., $A(\Cal T)\cong
C(\Cal T')$) if and only if $\Cal T$ is two-sided symmetric,
$2^{\infty}$ divides $\mathrm{Stz}(\Cal S')$, and the conditions
$(\Cal A_1),\,(\Cal A_2),\,(\Cal A_3)$ of \refth{BZh1} hold.

\item[(ii)] $O(\Cal T')\cong C(\Cal T'')$ if and only if
$2^{\infty}$ divides both $\mathrm{Stz}(\Cal S')$, and
$\mathrm{Stz}(\Cal S'')$, and the conditions $(\Cal A_1),\,(\Cal
A_2),\,(\Cal A_3)$ of \refth{BZh1} hold.
\end{itemize}

\end{theo}

\textbf{Remark.} It is easy to see from \refth{BZh1} that a
diagonal Lie algebra $X(\Cal T)$ is finitary (i.e. isomorphic to
$\sl(\infty)$, $\so(\infty)$, or $\sp(\infty)$) if and only if
$\mathrm{Stz}(\Cal S)$ is finite.

As we see from the above classification, the density type and the
symmetry type are well-defined for a diagonal Lie algebra. We will
call an algebra \emph{pure}, \emph{dense}, or \emph{sparse} if its
sequence of triples $\Cal T$ can be chosen pure, dense, or sparse,
respectively. We will also call an algebra \emph{one-sided},
\emph{two-sided symmetric}, \emph{two-sided strongly
non-symmetric}, or \emph{two-sided weakly non-symmetric} if its
sequence of triples $\Cal T$ can be chosen with the respective
property.

For an arbitrary sequence $\Cal S=\{s_i\}_{i\geq 1}$ by
$\sl(\mathrm{Stz}(\Cal S))$ (respectively, $\so(\mathrm{Stz}(\Cal
S))$, $\sp(\mathrm{Stz}(\Cal S))$) we will denote the pure Lie
algebra $A(\{(s_i,0,0)\}_{i\geq 1})$ (resp.,
$O(\{(s_i,0,0)\}_{i\geq 1})$, $C(\{(s_i,0,0)\}_{i\geq 1})$).

We need two branching rules for Lie algebras of type $A$.
Throughout this paper $F_n^\lambda$ denotes an irreducible
$\sl(n)$-module with highest weight
$\lambda=(\lambda_1,\dots,\lambda_n)$, $\lambda_i\in\ZZ_{\geq 0}$.
Note that the isomorphism class of $F_n^\lambda$ is determined by
the differences
$\lambda_1-\lambda_2,\dots,\lambda_{n-1}-\lambda_n$.

\begin{theo} (Gelfand-Tsetlin rule \cite{Z})\label{GTrule}
Consider a subalgebra $\sl(n)\subset\sl(n+1)$ of signature
$(1,0,1)$. Then, there is an isomorphism of $\sl(n)$-modules
\begin{equation}\label{GTform}
F_{n+1}^\lambda\downarrow\sl(n)\cong\bigoplus_\mu F_n^\mu,
\end{equation} where the summation runs over all integral weights
$\mu=(\mu_1,\dots,\mu_n)$ satisfying
$\lambda_1\geq\mu_1\geq\lambda_2\geq\cdots\geq\mu_n\geq\lambda_{n+1}$.
\end{theo}

Consider the $\sl(n)\oplus\sl(n)$-module $F_n^\mu\otimes F_n^\nu$.
By Theorem 2.1.1 of \cite{HTW} its restriction to
$\sl(n):=\{x\oplus x,\,x\in\sl(n)\}$ decomposes as
$\displaystyle\bigoplus_\lambda c_{\mu\nu}^\lambda F_n^\lambda$,
where $c_{\mu\nu}^\lambda$ is the Littlewood-Richardson
coefficient. One can iterate this branching rule to obtain the
decomposition for higher tensor products. Let
$c_{\mu_1\dots\mu_k}^\lambda$ denote the coefficient obtained in
this manner, so,
\begin{equation}\label{br1} F_n^{\mu_1}\otimes\cdots\otimes
F_n^{\mu_k}\downarrow\sl(n)\cong\bigoplus_\lambda
c_{\mu_1\dots\mu_k}^\lambda F_n^\lambda,
\end{equation} where the summation runs over all integral dominant weights $\lambda$ with $\lambda_i\geq 0$. We will call the numbers
$c_{\mu_1\dots\mu_k}^\lambda$ \emph{generalized
Littlewood-Richardson coefficients}.

The following branching rule was communicated to us by J.
Willenbring.
\begin{prop}\label{LRprop}
Consider a diagonal subalgebra $\sl(n)\subset\sl(kn)$ of signature
$(k,0,0)$. Then, there is an isomorphism of $\sl(n)$-modules
\begin{equation}\label{LRform}
\displaystyle F_{kn}^\lambda\downarrow\sl(n)\cong\bigoplus_\nu
(\sum_{\mu_1,\dots,\mu_k}c_{\mu_1\dots\mu_k}^\lambda
c_{\mu_1\dots\mu_k}^\nu)F_n^\nu, \end{equation} where one
summation runs over all integral dominant weights $\nu$ with
$\nu_i\geq 0$ for all $i$ and the other summation runs over all
sets of integral dominant weights $\mu_1,\dots,\mu_k$ with
$(\mu_j)_i\geq 0$ for all $i,j$.
\end{prop}
\begin{proof}
Consider the block-diagonal subalgebra
$\sl(l)\oplus\sl(m)\subset\sl(n)$ ($n=l+m$). By Theorem 2.2.1 of
\cite{HTW} $F_n^\lambda\downarrow\sl(l)\oplus\sl(m)$ decomposes as
$\displaystyle\bigoplus_{\mu\nu} c_{\mu\nu}^\lambda F_l^\mu\otimes
F_m^\nu$. Let now the direct sum of $k$ copies of $\sl(n)$ be a
subalgebra $\sl(kn)$ with block diagonal inclusion. By iteration
of this branching rule we see that the decomposition of
$F_{kn}^\lambda\downarrow\sl(n)\oplus\cdots\oplus\sl(n)$ is
determined by the generalized Littlewood-Richardson coefficients:
\begin{equation}\label{br2}
F_{kn}^\lambda\downarrow\sl(n)\oplus\cdots\oplus\sl(n)\cong\bigoplus_{\mu_1\dots\mu_k}
c_{\mu_1\dots\mu_k}^\lambda F_n^{\mu_1}\otimes\cdots\otimes
F_n^{\mu_k},
\end{equation}
where $sl(n)\oplus\cdots\oplus\sl(n)$ is the block-diagonal
subalgebra of $\sl(kn)$, and the summation runs over all integral
dominant weights $\mu_1,\dots,\mu_k$ with $(\mu_j)_i\geq 0$.

Consider now a subalgebra $\sl(n)\subset\sl(kn)$ of signature
$(k,0,0)$. One can obtain \refeq{LRform} as a combination of the
two branching rules \refeq{br1} and \refeq{br2}.
\end{proof}

\textbf{Remark.} In \refprop{LRprop} the sum is taken over all
integral dominant weights $\nu$ with $\nu_i\in\ZZ_{\geq 0}$ for
all $i$. In order for $F_n^\nu$ to have a non-zero coefficient in
\refeq{LRform} both Littlewood-Richardson coefficients
$c_{\mu_1\dots\mu_k}^\lambda$ and $c_{\mu_1\dots\mu_k}^\nu$ must
be non-zero for some $\mu_1,\dots,\mu_k$. But for that we must
have $\displaystyle\sum_{i=1}^{kn}\lambda_i=\sum_{i=1}^n\nu_i$.
Therefore the summation in \refeq{LRform} may be taken to run over
only those weights $\nu$ with fixed
$\displaystyle\sum_{i=1}^n\nu_i$. Hence all modules $F_n^\nu$
which are present in \refeq{LRform} with non-zero coefficients are
pairwise non-isomorphic. Indeed, if $F_n^{\nu'}\cong F_n^\nu$ both
have non-zero coefficients in \refeq{LRform}, then the weight
$\nu'$ can be obtained by shifting the weight $\nu$ by an integer,
so $\displaystyle\sum_{i=1}^n\nu_i=\sum_{i=1}^n\nu'_i$ implies
$\nu'=\nu$. This argument allows us to refer to a non-zero
coefficient
$\displaystyle(\sum_{\mu_1,\dots,\mu_k}c_{\mu_1\dots\mu_k}^\lambda
c_{\mu_1\dots\mu_k}^\nu)$ as the multiplicity of $F_n^\nu$ in
\refeq{LRform}.
\begin{corollary}\label{LRcor}
For a diagonal subalgebra $\sl(n)\subset\sl(kn)$ of signature
$(k,0,0)$ the restriction $F_{kn}^\lambda\downarrow\sl(n)$ has a
submodule with highest weight
$$(\nu_1,\dots,\nu_n)=(\lambda_1+\cdots+\lambda_k,\lambda_{k+1}+\cdots+\lambda_{2k},\dots,\lambda_{kn-k+1}+\cdots+\lambda_{kn}).$$
\end{corollary}
\begin{proof}
Indeed, if we set
$\mu_i=(\lambda_i,\lambda_{k+i},\dots,\lambda_{kn-k+i})$ for
$i\in\{1,\dots,k\}$, then it easy to check that both coefficients
$c_{\mu_1\dots\mu_k}^\lambda$ and $c_{\mu_1\dots\mu_k}^\nu$ are
non-zero, and therefore the highest weight module $F_n^\nu$ is
present in \refeq{LRform} with non-zero multiplicity.
\end{proof}

If $\ss$ and $\gg$ are two diagonal Lie algebras, then
constructing a homomorphism $\theta:\ss\rightarrow\gg$ is
equivalent to constructing commutative diagram
\begin{equation}\label{diag0}
\xymatrix{
\ss_1\ar[d]_{\theta_1}\ar[r]^{\phi_1}&\ss_2\ar[d]_{\theta_2}\ar[r]^{\phi_2}&\dots\\
\gg_1\ar[r]^{\psi_1}&\gg_2\ar[r]^{\psi_2}&\dots\\
}
\end{equation} for some exhaustions
$\displaystyle\ss_1\stackrel{\phi_1}{\rightarrow}\ss_2\stackrel{\phi_2}{\rightarrow}\dots$
and
$\displaystyle\gg_1\stackrel{\psi_1}{\rightarrow}\gg_2\stackrel{\psi_2}{\rightarrow}\dots$
of $\ss$ and $\gg$ respectively. An injective homomorphism
$\theta$ is called \emph{diagonal} if all $\theta_i$ can be chosen
diagonal for sufficiently large $i$.

To deal with diagonal homomorphisms we will need the following
result.

\begin{lemma}\label{BZhLemma}
Let $\varepsilon_1:\ss_1\rightarrow\ss_2$ and
$\varepsilon_2:\ss_1\rightarrow\gg$ be diagonal injective
homomorphisms of finite-dimensional simple classical Lie algebras
of signatures $(l,r,z)$ and $(p,q,u)$ respectively. Let a triple
of non-negative integers $(p',q',u')$ satisfy the following
conditions:
$$
p+q=(l+r)(p'+q'),\,p-q=(l-r)(p'-q'),\,n=n_2(p'+q')+u',
$$
where $n$ and $n_2$ are the dimensions of the natural $\gg$- and
$\ss_2$-modules respectively. Then, under the assumption that
$\ss_2$ and $\gg$ are of the same type $X$, there exists a
diagonal injective homomorphism $\theta:\ss_2\rightarrow\gg$ of
signature $(p',q',u')$ such that
$\varepsilon_2=\theta\circ\varepsilon_1$. If $\ss_2$ and $\gg$ are
of different types $X$ and $Y$, the statement holds under the
following additional conditions on the triple $(p',q',u')$:
\begin{align*}
&p'=q'\text{ if }(X,Y)=(A,O)\text{ or }(X,Y)=(A,C);\\&p'\text{ is
even if }(X,Y)=(O,C)\text{ or }(X,Y)=(C,O).
\end{align*}
\end{lemma}
\begin{proof}
Lemma 2.6 in \cite{BZ} states the same result in case all Lie
algebras $\ss_1$, $\ss_2$, $\gg$ are of the same type. The proof
of Lemma 2.6 in \cite{BZ} works also when the three algebras are
not of the same type, but only if $\ss_2$ can be mapped into $\gg$
by an injective homomorphism of signature $(p',q',u')$. It is easy
to check that the additional conditions guarantee the existence of
such a homomorphism.
\end{proof}

Consider the diagram in \refeq{diag0} without the commutativity
assumption. \refle{BZhLemma} implies that if all $\theta_i$ are
diagonal injective homomorphism such that for all $i\geq 1$ the
two diagonal injective homomorphisms $\psi_i\circ\theta_i$ and
$\theta_{i+1}\circ\phi_i$ of $\ss_i$ into $\gg_{i+1}$ have the
same signature, then there are diagonal injective homomorphisms
$\theta'_i$ with the same property making the diagram commutative.
Later on in this paper when constructing diagrams as in
\refeq{diag0} in concrete situations, we will check commutativity
by showing only that the signatures of $\psi_i\circ\theta_i$ and
$\theta_{i+1}\circ\phi_i$ coincide for all $i\geq 1$. It will then
be assumed that $\theta_i$ are replaced by corresponding diagonal
injective homomorphisms $\theta'_i$ making the diagram commute.

The following result can be found in \cite{BZ} (see also all
references in there, for instance \cite{B2}).

\begin{lemma}
Let $\hh\subset\gg\subset\ss$ be finite-dimensional classical
simple Lie algebras, $\rk\,\hh>10$. Assume that the inclusion
$\hh\subset\ss$ is diagonal. Then the inclusions $\hh\subset\gg$
and $\gg\subset\ss$ are also diagonal.
\end{lemma}
\begin{corollary}\label{diagCor}
Let $\hh\subset\gg\subset\ss$ be infinite-dimensional diagonal Lie
algebras. Assume that the inclusion $\hh\subset\ss$ is diagonal.
Then the inclusions $\hh\subset\gg$ and $\gg\subset\ss$ are also
diagonal.
\end{corollary}

We conclude this section by introducing a notion of equivalence of
infinite-dimensional Lie algebras. We say that $\gg_1$ is
\emph{equivalent} to $\gg_2$ ($\gg_1\sim\gg_2$) if there exist
injective homomorphisms $\gg_1\rightarrow\gg_2$ and
$\gg_2\rightarrow\gg_1$. For finite-dimensional Lie algebras this
equivalence relation is the same as isomorphism, but this is no
longer the case for infinite-dimensional Lie algebras.

\section{Classification of locally simple subalgebras of diagonal Lie algebras}
In this section all diagonal Lie algebras considered are assumed
to be infinite dimensional.

We start the classification by asking whether $\sl(\infty)$ admits
an injective homomorphism into any non-finitary diagonal Lie
algebra. As it turns out, the most basic example sufficed to
answer this question, as we were able to construct an injective
homomorphism of $\sl(\infty)$ into $\sl(2^\infty)$, so the answer
is yes. The following construction was suggested to us by I.
Dimitrov.

Let $F_n$ be the natural representation of $\sl(n)$. Note that
under the injective homomorphism $\sl(n)\to\sl(n+1)$ of signature
$(1,0,1)$, the exterior algebra $\bigwedge^{\cdot}(F_{n+1})$
decomposes as two copies of $\bigwedge^{\cdot}(F_n)$ as an
$\sl(n)$-module. Fix a map $\theta_n:\sl(n)\to\sl(2^n)$ such that
the natural representation of $\sl(2^n)$ decomposes as
$\bigwedge^{\cdot}(F_n)$ as an $\sl(n)$-module. Then there exists
a map $\theta_{n+1}:\sl(n+1)\to\sl(2^{n+1})$ such that the natural
representation of $\sl(2^{n+1})$ decomposes as
$\bigwedge^{\cdot}(F_{n+1})$ as an $\sl(n+1)$-module making the
following diagram commute:
\begin{equation}\label{diag1}
\xymatrix{
\sl(2)\ar[d]_{\theta_2}\ar[r]&\dots\ar[r]&\sl(n)\ar[d]_{\theta_n}\ar[r]&\sl(n+1)\ar[d]_{\theta_{n+1}}\ar[r]&\dots\\
\sl(2^2)\ar[r]&\dots\ar[r]&\sl(2^n)\ar[r]&\sl(2^{n+1})\ar[r]&\dots,\\
}
\end{equation} where the lower row consists of injective homomorphisms of
signature $(2,0,0)$. Therefore by induction, the diagram yields an
injective homomorphism of $\sl(\infty)$ into $\sl(2^{\infty})$.

We will prove now that similar injective homomorphisms exist in a
more general setting. The following result will be used later to
prove that in fact any finitary diagonal Lie algebra can be
similarly mapped into any diagonal Lie algebra.
\begin{prop}\label{prop1} $\sl(\infty)$ admits an injective homomorphism into any pure one-sided Lie algebra $\ss$ of type $A$.
\end{prop}
\begin{proof}
By \refth{BZh1} $\ss$ is isomorphic to $\sl(\Pi)$ for some
infinite Steinitz number $\Pi$. Then it is sufficient to show the
existence of a commutative diagram
\begin{equation}\label{diag2}
\xymatrix{
\sl(2)\ar[d]_{\theta_2}\ar[r]&\sl(3)\ar[d]_{\theta_3}\ar[r]&\dots\ar[r]&\sl(k)\ar[d]_{\theta_k}\ar[r]&\sl(k+1)\ar[d]_{\theta_{k+1}}\ar[r]&\dots\\
\sl(n_1 n_2)\ar[r]&\sl(n_1 n_2
n_3)\ar[r]&\dots\ar[r]&\sl(n_1\cdots n_k)\ar[r]&\sl(n_1\cdots
n_{k+1})\ar[r]&\dots }
\end{equation}
for suitable $\{n_i\}$, where $\theta_i$ are injective
homomorphisms and $n_1,n_2,\dots$ are chosen so that
$\displaystyle\prod_{i=1}^\infty n_i=\Pi$. Indeed, the diagram in
\refeq{diag2} yields an injective homomorphism
$\sl(\infty)\rightarrow\sl(n_1n_2\cdots)$, and $\sl(n_1n_2\cdots)$
is isomorphic to $\ss$ by \refth{BZh1}. We will choose the
homomorphisms $\theta_k$ so that
$$V_k\downarrow\sl(k)\cong a_0^k\bigwedge^0(F_k)\oplus
a_1^k\bigwedge^1(F_k)\oplus\cdots\oplus a_k^k\bigwedge^k(F_k)$$ as
$\sl(k)$-modules. Here $V_k$ stands for the natural $\sl(n_1\cdots
n_k)$-module, $F_k$ is the natural $\sl(k)$-module and the
coefficients $a_i^k$, $i=0,\dots,k$ are non-negative integers. The
above injective homomorphism of $\sl(\infty)$ into $\sl(2^\infty)$
corresponds to the particular case $n_k=2$ and $a_i^k=1$ for all
$k\geq 2,\,i=0,\dots,k$.

We see that if the numbers $\{a_i^k\}$ satisfy the conditions
$a_i^k+a_{i+1}^k=n_k a_i^{k-1},\,k\geq 3,\,i=0,\dots,k-1$ and
$a_0^2+2a_1^2+a_2^2=n_1 n_2$, then the homomorphisms $\theta_k$
can be chosen so that the diagram in \refeq{diag2} commutes.

We will add numbers $a_0^1,a_1^1,a_0^0$ to the set of coefficients
$\{a_i^k\}$ and will require $a_0^2+a_1^2=n_2a_0^1$,
$a_1^2+a_2^2=n_2a_1^1$, $a_0^1+a_1^1=n_1$, and $a_0^0=1$. Then the
numbers $\{a_i^k\}$ will form an infinite triangle
$$
\begin{array}{ll}
\;\;\;\;\;\;\;a_0^0\\
\;\;\;\;\;a_0^1\;a_1^1\\
\;\;\;a_0^2\;a_1^2\;a_2^2\\
\;\;\;\;\;\;\;\dots\\
\end{array}
$$
such that \begin{equation}\label{cond1} a_i^k+a_{i+1}^k=n_k
a_i^{k-1},\,k\geq 1\text{ and }a_0^0=1.\end{equation}

It is enough to prove that a triangle of non-negative integers
satisfying \refeq{cond1} exists for a suitable choice of $n_i$.
Set $b_k:=\frac{a_{k-1}^k}{n_1\cdots n_k}$ for $k\geq 1$. A simple
calculation shows that $a_k^k=n_1\cdots
n_k(a_0^0-b_1-b_2-\cdots-b_k)$. Notice that since $a_0^0=1$, the
numbers $b_1,b_2,\dots$ uniquely determine the entire triangle, as
the $l^{\text{th}}$ ``diagonal'' $\{a_k^{k+l}\}_{k\geq 0}$ of the
triangle is determined by the previous diagonal
$\{a_k^{k+l-1}\}_{k\geq 0}$ and the sequence $n_1,n_2,\dots$.

Now we will find conditions on $b_k$ under which all $a_i^k$ will
be non-negative. Since $a_k^{k+1}\geq 0$, the numbers $b_k$ should
be non-negative. In order for $a_k^k$ to be non-negative we should
have $\displaystyle\sum_{i=1}^k b_i<a_0^0$ for all $k$ (since
$b_i$ are non-negative, we can rewrite these conditions as
$\displaystyle\sum_{i=1}^\infty b_i\leq 1$). The entries of the
diagonal $\{a_k^{k+2}\}_{k\geq 0}$ can be found from
\refeq{cond1}: $a_k^{k+2}=n_1\cdots n_{k+2}(b_{k+1}-b_{k+2})$ for
$k\geq 0$. This requires the sequence $\{b_k-b_{k+1}\}$ to be
non-negative. If we set $b_k^{(1)}:=b_k-b_{k+1}$ for $k\geq 1$,
then in a similar way we obtain $a_k^{k+3}=n_1\cdots
n_{k+3}(b_{k+1}^{(1)}-b_{k+2}^{(1)})$. This requires the sequence
$\{b_k^{(2)}:=b_k^{(1)}-b_{k+1}^{(1)}\}$ to be non-negative.
Continuing this procedure, we get $a_k^{k+l}=n_1\cdots
n_{k+l}b_{k+1}^{(l-1)}$ for all $l\geq 3$, where by definition
$b_k^{(l+1)}=b_k^{(l)}-b_{k+1}^{(l)}$. Now we see that the
non-negative integers $a_i^k$ satisfying \refeq{cond1} exist if
there exists a non-negative sequence $\{b_k\}_{k\geq 1}$ with
$b_k\in\frac{1}{n_1\cdots n_k}\ZZ_{\geq 0}$ and
$\displaystyle\sum_{k=1}^\infty b_k\leq 1$ such that
\begin{equation}\label{cond2} \text{all iterated sequences of
differences }\{b_k^{(l)}\}_{k\geq 1}\text{ are non-negative.}
\end{equation}

Note that the sequence $\{b_k=\frac{1}{q^k}\},\,q>1$ satisfies
\refeq{cond2} as $b_k^{(l)}=\frac{1}{q^k}(1-\frac{1}{q})^l>0$ for
all $k,l\geq 1$. (In the case $n_k=n$ for all $k$, taking $q=n$
yields an injective homomorphism
$\sl(\infty)\hookrightarrow\sl(n^\infty)$.) We will find the
desired sequence $\{b_k\}$ as a convergent infinite linear
combination of geometric sequences.

Let us put $q=4$ (the following construction would work for any
$q\geq 4$) and let $\Pi=m_1 m_2\cdots$. Choose a strictly
increasing sequence of integers $\{l_k\}_{k\geq 0}$ so that
$l_0=0$ and $m_1 m_2\cdots m_{l_k}>\frac{(q-1)q^{k^2+1}}{q-2}$ for
$k\geq 1$, which is possible as $\Pi$ is infinite. Take
$n_k=m_{l_{k-1}+1}\cdots m_{l_k}$ for $k\geq 1$. Then clearly $n_1
n_2\cdots=\Pi$.

Let us now construct the sequence $\{b_k\}$ for the chosen
$n_1,n_2,\dots$. For $i\geq 1$ we denote $\displaystyle c_i=1+
\sum_{j=i}^\infty\frac{\varepsilon_j}{\frac{1}{q^i}(\frac{1}{q^i}-\frac{1}{q})\cdots(\frac{1}{q^i}-\frac{1}{q^{i-1}})(\frac{1}{q^i}-\frac{1}{q^{i+1}})\cdots(\frac{1}{q^i}-\frac{1}{q^j})}$,
where the numbers $\varepsilon_j$, satisfying
\begin{equation}\label{cond3}
0\leq\varepsilon_j<\frac{q-2}{(q-1)q^{j^2+1}}, \end{equation} are
to be chosen later, and put $\displaystyle b_k=\sum_{i=1}^\infty
c_i \left(\frac{1}{q^i}\right)^k$. We will show that for the
numbers $\varepsilon_j$, satisfying \refeq{cond3}, the series for
$c_i$ converges to a positive number for $i\geq 1$, the series for
$b_k$ converges for $k\geq 1$, and $\displaystyle\sum_{k=1}^\infty
b_k\leq 1$. Moreover, we will show that by varying $\varepsilon_j$
inside corresponding intervals we can make each $b_k$ to be of the
form $\frac{1}{n_1\cdots n_k}\ZZ_{\geq 0}$. We will have then
$\displaystyle b_k^{(l)}=\sum_{i=1}^\infty c_i
\left(\frac{1}{q^i}\right)^k\left(1-\frac{1}{q^i}\right)^l\geq 0$,
so $\{b_k^{(l)}\}$ will be a sequence of non-negative numbers for
any $l$. Hence the final condition in \refeq{cond2} will be
satisfied.

As a matter of convenience we denote $q_i=\frac{1}{q^i}$. Then let
$c_{ij}=\frac{\varepsilon_j}{q_i(q_i-q_1)\cdots(q_i-q_{i-1})(q_i-q_{i+1})\cdots(q_i-q_j)}$
for $i\leq j$. We see that $\displaystyle c_i=1+\sum_{j=i}^\infty
c_{ij}$. Let us prove that this series converges absolutely. We
have \begin{align*}
|c_i-1|&=\left|\sum_{j=i}^\infty\frac{\varepsilon_j}{(\frac{1}{q^i})^j(1-q^{i-1})\cdots(1-q)(1-\frac{1}{q})\cdots(1-\frac{1}{q^{j-i}})}\right|\\
&\leq\sum_{j=i}^\infty\frac{\varepsilon_j}{(\frac{1}{q^i})^j(q^{i-1}-1)\cdots(q-1)(1-\frac{1}{q})\cdots(1-\frac{1}{q^{j-i}})}\\
&\leq \sum_{j=i}^\infty\frac{\varepsilon_j
q^{ij}}{(1-\frac{1}{q})(1-\frac{1}{q^2})\cdots}
\leq\sum_{j=i}^\infty\frac{\varepsilon_j
q^{ij}}{(1-\frac{1}{q}-\frac{1}{q^2}-\cdots)}=
\sum_{j=i}^\infty\frac{\varepsilon_j
q^{ij}(q-1)}{q-2}.\end{align*} Then, using \refeq{cond3}, we
obtain $\displaystyle
|c_i-1|\leq\sum_{j=i}^\infty\frac{q^{ij}}{q^{j^2+1}}=
\frac{1}{q}+\frac{1}{q^{i+2}}+\frac{1}{q^{2i+5}}+\cdots<
\frac{1}{q}+\frac{1}{q^2}+\cdots=\frac{1}{q-1}$. Thus, the series
$\displaystyle 1+\sum_{j=i}^\infty c_{ij}$ converges absolutely
and its sum $c_i$ is a number from the interval
$\left(\frac{q-2}{q-1},\frac{q}{q-1}\right)$ (in particular, $c_i$
is positive) for all $i$. Furthermore, \begin{align*}
\sum_{k=1}^\infty b_k&=\sum_{i=1}^\infty
\frac{c_i}{q^i}+\sum_{i=1}^\infty \frac{c_i}{(q^2)^i}+\cdots <
\frac{q}{q-1}\left(\sum_{i=1}^\infty\frac{1}{q^i}+\sum_{i=1}^\infty\frac{1}{(q^2)^i}+\cdots\right)\\
&=
\frac{q}{q-1}\left(\frac{1}{q-1}+\frac{1}{q^2-1}+\frac{1}{q^3-1}+\cdots\right)\\
&<
\frac{q}{q-1}\left(\frac{1}{q-1}+\frac{1}{(q-1)^2}+\cdots\right)=
\frac{q}{q-1}\cdot\frac{1}{q-2}<1\textrm{ because }q\geq
4.\end{align*} Since every term in these expressions is
non-negative, the convergence of each series $\displaystyle
b_k=\sum_{i=1}^\infty c_i \left(\frac{1}{q^i}\right)^k$ follows.

Finally, let us show that the numbers $\varepsilon_j$, satisfying
\refeq{cond3}, can be chosen so that $b_k\in\frac{1}{n_1\cdots
n_k}\ZZ_{\geq 0}$. We know that $\displaystyle
b_k=\sum_{i=1}^\infty c_i q_i^k=\sum_{i=1}^\infty
q_i^k+\sum_{i=1}^\infty\sum_{j=i}^\infty c_{ij} q_i^k$. From what
we proved it follows that the latter sum is absolutely convergent.
Therefore we can rewrite it as $\displaystyle
b_k=\sum_{i=1}^\infty q_i^k+\sum_{j=1}^\infty\sum_{i=1}^j
c_{ij}q_i^k$. Note that the numbers $c_{ij}$ were defined as
solutions of the equation
$\left(\begin{array}{ll}q_1\;\;\;\;\;\dots\;
q_j\\\;\vdots\;\;\;\;\;\;\ddots\;\;\vdots\\q_1^{j-1}\;\dots\;
q_j^{j-1}\\q_1^j\;\;\;\;\,\dots\; q_j^j\\\end{array}\right)
\left(\begin{array}{ll}c_{1j}\\\;\vdots\\c_{jj}\\\end{array}\right)=
\left(\begin{array}{ll}0\\\vdots\\0\\\varepsilon_j\\\end{array}\right)$
using the well-known formula for inverting a Vandermonde matrix.
Thus, $\displaystyle\sum_{i=1}^j q_i^k c_{ij}=0$ for $k<j$ and
$\displaystyle\sum_{i=1}^j q_i^j c_{ij}=\varepsilon_j$. Hence,
$\displaystyle b_k=\sum_{i=1}^\infty
q_i^k+\sum_{j=1}^{k-1}\sum_{i=1}^j c_{ij}q_i^k+\varepsilon_k$, so
$b_k-\varepsilon_k$ depends only on
$\varepsilon_1,\dots,\varepsilon_{k-1}$. Let us introduce the
notation $\displaystyle
f_k(\varepsilon_1,\dots,\varepsilon_{k-1})=\sum_{i=1}^\infty
q_i^k+\sum_{j=1}^{k-1}\sum_{i=1}^j c_{ij}q_i^k$ for $k\geq 2$ and
$\displaystyle f_1=\sum_{i=1}^\infty
q_i=\sum_{i=1}^\infty\frac{1}{q^i}=\frac{1}{q-1}$.

Now we define inductively the numbers $\varepsilon_k$. We choose
$\varepsilon_1$ in such a way that $b_1$ is the smallest number of
the form $\frac{1}{n_1}\ZZ_{\geq 0}$ which is not less than $f_1$.
Then we have
$0\leq\varepsilon_1=b_1-f_1<\frac{1}{n_1}<\frac{q-2}{(q-1)q^2}$
(because of the choice of $n_1$), so $\varepsilon_1$ lies inside
the corresponding interval in \refeq{cond3}. For fixed
$\varepsilon_1,\dots,\varepsilon_{k-1}$ we choose $\varepsilon_k$
to make $b_k$ the smallest number of the form $\frac{1}{n_1\cdots
n_k}\ZZ_{\geq 0}$ which is not less than
$f_k(\varepsilon_1,\dots,\varepsilon_{k-1})$. Then
$0\leq\varepsilon_k=b_k-f_k(\varepsilon_1,\dots,\varepsilon_{k-1})<\frac{1}{n_1\cdots
n_k}<\frac{q-2}{(q-1)q^{k^2+1}}$ (again, because of the choice of
$n_1,\dots,n_k$), so $\varepsilon_k$ satisfies \refeq{cond3}.
Therefore the sequence $\{b_k\}$ satisfies all the required
conditions, and the statement follows.
\end{proof}

\textbf{Remark.} Since $\so(\infty)$ and $\sp(\infty)$ are
subalgebras of $\sl(\infty)$, each of them admits also an
injective homomorphism into any one-sided pure diagonal Lie
algebra of type $A$.

The following two lemmas show that certain conditions guarantee
the existence of injective homomorphisms of non-finitary diagonal
Lie algebras.

\begin{lemma}\label{lemma1}
Let $\ss_1=X(\Cal T_1)$ and $\ss_2=X(\Cal T_2)$ be diagonal Lie
algebras of the same type ($X=A$, $C$, or $O$), neither of them
finitary. Set $S_i=\mathrm{Stz}(\Cal S_i)$,
$S=\mathrm{GCD}(S_1,S_2)$, $R_i=\quotst(S_i,S)$,
$\delta_i=\delta(\Cal T_i)$, $C_i=\mathrm{Stz}(\Cal C_i)$,
$C=\mathrm{GCD}(C_1,C_2)$, $B_i=\quotst(C_i,C)$, and
$\sigma_i=\sigma(\Cal T_i)$ for $i=1,2$. We assume that $R_1$ is
finite.
\begin{itemize}
\item[(i)] Assume that $\ss_1$ and $\ss_2$ are non-sparse of type
$A$, both $R_1$ and $R_2$ are finite, and $S$ is not divisible by
an infinite power of any prime number. If
$2\frac{R_1}{\delta_1}<\frac{R_2}{\delta_2}$, then $\ss_1$ admits
an injective homomorphism into $\ss_2$. If
$2\frac{R_1}{\delta_1}=\frac{R_2}{\delta_2}$, $\ss_1$ admits an
injective homomorphism into $\ss_2$ unless $\ss_1$ is pure and
$\ss_2$ is dense.

\item[(ii)] Assume that $\ss_1$ and $\ss_2$ are non-sparse, both
$R_1$ and $R_2$ are finite, and $S$ is not divisible by an
infinite power of any prime number. In addition, assume that one
of the following is true:
\begin{itemize}
\item both $\ss_1$ and $\ss_2$ are one-sided; \item $B_1$ is
finite, either $\ss_1$ is one-sided and $\ss_2$ is two-sided
non-symmetric or $\ss_2$ is two-sided weakly non-symmetric and
$\ss_1$ is two-sided non-symmetric; \item $B_1$ is finite, both
$\ss_1$ and $\ss_2$ are two-sided strongly non-symmetric, either
$B_2$ is infinite or $C$ is divisible by an infinite power of some
prime number; \item both $B_1$ and $B_2$ are finite, both $\ss_1$
and $\ss_2$ are two-sided strongly non-symmetric, $C$ is not
divisible by an infinite power of a prime number, and
$\frac{R_1\sigma_1}{B_1}\geq\frac{R_2\sigma_2}{B_2}$.
\end{itemize} Then, if
$\frac{R_1}{\delta_1}<\frac{R_2}{\delta_2}$, $\ss_1$ admits an
injective homomorphism into $\ss_2$. If
$\frac{R_1}{\delta_1}=\frac{R_2}{\delta_2}$, $\ss_1$ admits an
injective homomorphism into $\ss_2$ unless $\ss_1$ is pure and
$\ss_2$ is dense.

\item[(iii)] Assume that $\ss_1$ and $\ss_2$ are non-sparse. If
$R_2$ is infinite or $S$ is divisible by an infinite power of some
prime number, then $\ss_1$ admits an injective homomorphism into
$\ss_2$.

\item[(iv)] If $\ss_2$ is sparse, then $\ss_1$ admits an injective
homomorphism into $\ss_2$.
\end{itemize}
\end{lemma}
\begin{proof} The Steinitz numbers $S_1$, $C_1$ and the indices $\delta_1$,
$\sigma_1$ are in general not well-defined for a Lie algebra
$\ss_1$: these values characterize a given exhaustion of $\ss_1$.
However, if $\ss_1$ is non-sparse and $S_1$ is not divisible by an
infinite power of any prime number, then the number
$\frac{R_1}{\delta_1}$ does not depend on the exhaustion of
$\ss_1$ (because then by condition $\Cal A_2$ of \refth{BZh1}
$\frac{\mathrm{Stz}(\Cal S_1)}{\mathrm{Stz}(\Cal S'_1)}$ is a set
containing exactly one element for $S'_1$ corresponding to any
other exhaustion of $\ss_1$, and therefore $\frac{R_1}{\delta_1}$
is well-defined by condition $\Cal A_3$). Also, under the
assumptions made in the last statement of (ii) the number
$\frac{\sigma_1 R_1}{B_1}$ does not depend on the exhaustion of
$\ss_1$ (this follows from condition $\Cal B_3$ of \refth{BZh1}).
The finiteness of $R_1$, $R_2$, $B_1$, $B_2$ does not depend on
the exhaustion either, so in the proofs of all the statements we
can exhaust $\ss_1$ in any convenient way. The same applies to
$\ss_2$.

We will assume that $X=A$ and prove all four statements for type
$A$ Lie algebras. If $\ss_1$ and $\ss_2$ are of type $O$ or $C$,
then both $\ss_1$ and $\ss_2$ are one-sided and the proof is
analogous to the proof in the type $A$ case when $\ss_1$ and
$\ss_2$ are one-sided.

Let us now set up the notations for the proof of all four
statements. Let $\ss_1$ be exhausted as
$\sl(n_0)\subset\sl(n_1)\subset\cdots$, each inclusion
$\sl(n_i)\rightarrow\sl(n_{i+1})$ being of signature
$(l_i,r_i,z_i)$, $i\geq 0$. By possibly changing some first terms
of the exhaustion, we can choose $n_0$ to be divisible by $R_1$.
Similarly, let $\sl(m_0)\subset\sl(m_1)\subset\cdots$ be the
exhaustion of $\ss_2$, each inclusion
$\sl(m_i)\rightarrow\sl(m_{i+1})$ being of signature
$(l'_i,r'_i,z'_i)$, $i\geq 0$. Set $s_i=l_i+r_i$, $c_i=l_i-r_i$,
$s'_i=l'_i+r'_i$, and $c'_i=l'_i-r'_i$ for $i\geq 0$. Then
$S_1=n_0 s_0 s_1\cdots$, $C_1=n_0 c_0 c_1\cdots$, $S_2=m_0 s'_0
s'_1\cdots$, $C_2=m_0 c'_0 c'_1\cdots$,
$\displaystyle\delta_1=\lim_{i\rightarrow\infty}\frac{n_0s_0\cdots
s_{i-1}}{n_i}$,
$\displaystyle\delta_2=\lim_{i\rightarrow\infty}\frac{m_0s'_0\cdots
s'_{i-1}}{m_i}$,
$\displaystyle\sigma_1=\lim_{i\rightarrow\infty}\frac{c_0\cdots
c_i}{s_0\cdots s_i}$, and
$\displaystyle\sigma_2=\lim_{i\rightarrow\infty}\frac{c'_0\cdots
c'_i}{s'_0\cdots s'_i}$.

Consider a diagram
\begin{equation}\label{diag10}
\xymatrix{
\sl(n_0)\ar[d]_{\theta_0}\ar[r]&\sl(n_1)\ar[d]_{\theta_1}\ar[r]&\dots\ar[r]&\sl(n_i)\ar[d]_{\theta_i}\ar[r]&\sl(n_{i+1})\ar[d]_{\theta_{i+1}}\ar[r]&\dots\\
\sl(m_{k_0})\ar[r]&\sl(m_{k_1})\ar[r]&\dots\ar[r]&\sl(m_{k_i})\ar[r]&\sl(m_{k_{i+1}})\ar[r]&\dots,}
\end{equation}
where $\theta_i$ is a diagonal homomorphism of signature
$(x_i,y_i,m_{k_i}-(x_i+y_i)n_i)$, $i\geq 0$. Taking into
consideration our remark at the end of section 2, we see that to
make such a diagram well-defined and commutative it is enough to
have
\begin{equation}\label{cond4}
s_i(x_{i+1}+y_{i+1})=(x_i+y_i)s'_{k_i}\cdots
s'_{k_{i+1}-1},\end{equation}\begin{equation}\label{cond5}
c_i(x_{i+1}-y_{i+1})=(x_i-y_i)c'_{k_i}\cdots
c'_{k_{i+1}-1},\end{equation} and
\begin{equation}\label{cond6}m_{k_i}\geq(x_i+y_i)n_i\end{equation} for $i\geq 0$. Finally, we set $p_0=\frac{n_0}{R_1}$
and $p_i=p_0 s_0\cdots s_{i-1}$ for $i\geq 1$. We are now ready to
prove that there exist numbers $x_i,y_i,\,i\geq 0$ satisfying
\refeq{cond4} $-$ \refeq{cond6} in all four cases.

(i) The Steinitz number $R_2$ is finite in this case. Possibly by
changing the exhaustion of $\ss_2$ we can choose $m_0$ to be
divisible by $R_2$. Choose also each $k_i$ large enough so that
$m_0 s'_0\cdots s'_{k_i-1}$ is divisible by $R_2p_i$ (this is
possible since $p_i$ divides $S$) and put $q_i=\frac{m_0
s'_0\cdots s'_{k_i-1}}{R_2 p_i}$ for $i\geq 0$. Put $x_i=y_i=q_i$.
Then it is easy to verify that \refeq{cond4} and \refeq{cond5}
hold, and \refeq{cond6} is equivalent to $\frac{m_0 s'_0\cdots
s'_{k_i-1}}{R_2 m_{k_i}}\leq\frac{n_0 s_0\cdots s_{i-1}}{2R_1
n_i}$.

Suppose that $\frac{\delta_2}{R_2}<\frac{\delta_1}{2R_1}$. Pick
$\alpha\in(\frac{\delta_2}{R_2},\frac{\delta_1}{2R_1})$. Since
$\displaystyle\delta_1=\lim_{i\rightarrow\infty}\frac{n_0
s_0\cdots s_{i-1}}{n_i}$ and
$\displaystyle\delta_2=\lim_{i\rightarrow\infty}\frac{m_0
s'_0\cdots s'_i}{m_i}$ we have $\frac{m_0 s'_0\cdots
s'_{k_i-1}}{R_2 m_{k_i}}\leq\alpha\leq\frac{n_0 s_0\cdots
s_{i-1}}{2R_1 n_i}$ for $i\geq i_0$, $k_i\geq j_0$. Obviously we
can choose each $k_i$ greater than $j_0$. Also we can construct
$\theta_i$ only for $i\geq i_0$ and the diagram in \refeq{diag10}
will still give us an injective homomorphism of $\ss_1$ into
$\ss_2$.

Let now $\frac{\delta_2}{R_2}=\frac{\delta_1}{2R_1}$. If $\ss_2$
is pure then $\frac{m_0 s'_0\cdots s'_{k_i-1}}{R_2
m_{k_i}}=\frac{\delta_2}{R_2}=\frac{\delta_1}{2R_1}\leq\frac{n_0
s_0\cdots s_{i-1}}{2R_1 n_i}$, where the latter inequality holds
because the sequence $\frac{n_0 s_0\cdots s_{i-1}}{n_i}$ is
decreasing. Finally, if both $\ss_1$ and $\ss_2$ are dense, then
for each $i$ we have
$\frac{\delta_2}{R_2}=\frac{\delta_1}{2R_1}<\frac{n_0 s_0\cdots
s_{i-1}}{2R_1 n_i}$, so to make $\frac{m_0 s'_0\cdots
s'_{k_i-1}}{R_2 m_{k_i}}\leq\frac{n_0 s_0\cdots s_{i-1}}{2R_1
n_i}$ we choose $k_i$ sufficiently large.

(ii) Possibly by changing the exhaustions of $\ss_1$ and $\ss_2$
we choose $n_0$ to be divisible by $R_1 2^u$ and $m_0$ to be
divisible by $R_2 2^u$, where $u$ is the maximal power of 2
dividing $S$ ($u$ is finite because $2^\infty$ does not divide
$S$). We also choose $m_0$ large enough so that
$\frac{m_0}{R_2}\geq\frac{n_0}{R_1}$. Denote again $q_i=\frac{m_0
s'_0\cdots s'_{k_i-1}}{R_2 p_i}$, $i\geq 0$ ($k_i$ is chosen large
enough to make $R_2 p_i$ divide $m_0 s'_0\cdots s'_{k_i-1}$).

If both $\ss_1$ and $\ss_2$ are one-sided, we put $x_i=q_i,
y_i=0$. In the other three cases $B_1$ is finite, so $c_0
c_1\cdots$ divides $M c'_0 c'_1\cdots$ for some finite $M$. By
changing the exhaustion of $\ss_1$ we can make $c_0 c_1\cdots$
divide $c'_0 c'_1\cdots$. For that we replace the signature
$(l_i,r_i,z_i)$ with $((l_i+r_i+1)/2,(l_i+r_i-1)/2,z_i)$ for
finitely many $i$ ($l_i+r_i$ is odd for all $i\geq 0$ because $s_0
s_1\cdots=\frac{R_1S}{n_0}$ is not divisible by 2). Now we can
choose each $k_i$ large enough so that $c_0\cdots c_{i-1}$ divides
$c'_0\cdots c'_{k_i-1}$. Then denote $t_i=\frac{c'_0\cdots
c'_{k_i-1}}{c_0\cdots c_{i-1}}$ for $i\geq 1$ and $t_0=1$. Notice
that for each $i\geq 0$ the numbers $c_i$ and $c'_i$ have the same
parities as the numbers $s_i$ and $s'_i$ respectively. But all
$s_i$ and $s'_i$ are odd, so $c_i$ and $c'_i$ are odd as well.
Hence $t_i$ and $q_i$ are odd, and we put $x_i=(q_i+t_i)/2$ and
$y_i=(q_i-t_i)/2$. Let us check that $y_i\geq 0$ (or $q_i\geq
t_i$). This is obvious for $i=0$. For $i\geq 1$ the inequality
$y_i\geq 0$ is equivalent to $\frac{R_2}{m_0}\cdot\frac{c'_0\cdots
c'_{k_i-1}}{s'_0\cdots
s'_{k_i-1}}\leq\frac{R_1}{n_0}\cdot\frac{c_0\cdots
c_{i-1}}{s_0\cdots s_{i-1}}$, or
\begin{equation}\label{eq8}
\frac{R_2}{m_0}(\sigma_2)_{k_i}\leq\frac{R_1}{n_0}(\sigma_1)_i,
\end{equation} where $(\sigma_1)_i=\frac{c_0\cdots c_{i-1}}{s_0\cdots
s_{i-1}}$ is a decreasing sequence which tends to $\sigma_1$ and
$(\sigma_2)_i=\frac{c'_0\cdots c'_{i-1}}{s'_0\cdots s'_{i-1}}$ is
a decreasing sequence which tends to $\sigma_2$. Let us verify the
inequality in \refeq{eq8} case by case.

If $\ss_1$ is one-sided, then $(\sigma_1)_i=1$ for $i\geq 1$ and
our inequality is equivalent to $(\sigma_2)_{k_i}\leq\frac{m_0
R_1}{n_0 R_2}$. This holds in case $\ss_2$ is two-sided
non-symmetric because of the assumption
$\frac{m_0}{R_2}\geq\frac{n_0}{R_1}$ made at the beginning of the
proof. If $\ss_2$ is two-sided weakly non-symmetric, then
$\displaystyle\lim_{i\rightarrow\infty}(\sigma_2)_{k_i}=\sigma_2=0$,
and therefore $(\sigma_2)_{k_i}\leq\frac{m_0 R_1}{n_0
R_2}(\sigma_1)_i$ for large enough $k_i$ in case $\ss_1$ is
two-sided non-symmetric.

Let now both $\ss_1$ and $\ss_2$ be two-sided strongly
non-symmetric, $B_2$ be infinite or $C$ be divisible by an
infinite power of some prime number. In this case there exists an
infinite Steinitz number $C'$ such that $c_0 c_1\cdots$ divides
$\frac{1}{C'}c'_0 c'_1\cdots$. Since
$\displaystyle\sigma_1=\lim_{i\rightarrow\infty}(\sigma_1)_i>0$
and the sequence $(\sigma_1)_i$ decreases, to verify \refeq{eq8}
it suffices to prove that
$(\sigma_2)_{k_i}\leq\frac{m_0R_1}{n_0R_2}\sigma_1$. We have
$\frac{m_0}{R_2}\geq\frac{n_0}{R_1}$, therefore it is enough to
prove that $(\sigma_2)_{k_i}\leq\sigma_1$. This clearly holds for
large enough $k_i$ if $\sigma_2<\sigma_1$. Otherwise we change the
exhaustion of $\ss_2$ such that the new symmetry index
$\tilde{\sigma}_2=\sigma_2/N$ is less than $\sigma_1$ for a finite
$N|C'$ (we replace $l'_i,r'_i$ by $(s'_i+u)/2,(s'_i-u)/2$
respectively, where $c'_i=uv$ and $v|N$ for finitely many $i$) and
repeat the same construction of $x_i,y_i$. Then $\sigma_1$ stays
the same and in the new construction the inequality
$(\tilde{\sigma}_2)_{k_i}\leq \sigma_1$ holds for large enough
$k_i$.

Finally, let both $B_1$ and $B_2$ be finite, both $\ss_1$ and
$\ss_2$ be two-sided strongly non-symmetric, $C$ be not divisible
by an infinite power of a prime number, and
$\frac{R_1\sigma_1}{B_1}\geq\frac{R_2\sigma_2}{B_2}$. Then $c'_0
c'_1\cdots=N c_0 c_1\cdots$ for an odd number $N$, and by possibly
changing the exhaustion of $\ss_2$ we can make $c'_0
c'_1\cdots=c_0 c_1\cdots$ and repeat the same construction. Then
$\frac{B_1}{B_2}=\frac{n_0}{m_0}$, and therefore
$\frac{R_1\sigma_1}{R_2\sigma_2}\geq\frac{B_1}{B_2}=\frac{n_0}{m_0}$.
Then
$\displaystyle\lim_{i\rightarrow\infty}(\sigma_2)_{k_i}=\sigma_2<\frac{m_0
R_1}{n_0 R_2}(\sigma_1)_i$ for all $i$, since $(\sigma_1)_i$ is a
decreasing sequence which does not stabilize. Now clearly
\refeq{eq8} holds for large enough $k_i$.

So far we have proven that in all cases we can choose exhaustions
of $\ss_1$ and $\ss_2$ such that $x_i=\frac{1}{2}(q_i+t_i)$ and
$y_i=\frac{1}{2}(q_i-t_i)$ are non-negative integers (in the first
case, where both $\ss_1$ and $\ss_2$ are one-sided, we just put
$t_i=q_i$, so $x_i=q_i$, $y_i=0$). Since we have $x_i+y_i=q_i$ and
$x_i-y_i=t_i$, it is easy to check \refeq{cond4} and
\refeq{cond5}. The condition in \refeq{cond6} is equivalent to
$\frac{m_0 s'_0\cdots s'_{k_i-1}}{R_2 m_{k_i}}\leq\frac{n_0
s_0\cdots s_{i-1}}{R_1 n_i}$, and under the assumption
$\frac{\delta_2}{R_2}<\frac{\delta_1}{R_1}$ or
$\frac{\delta_2}{R_2}=\frac{\delta_1}{R_1}$ its proof is analogous
to that in (i).

(iii) Let us fix an exhaustion of $\ss_1$ and choose $m_0$ in the
exhaustion of $\ss_2$ such that $R'_2 p_0|m_0$ and
$\frac{m_0}{R'_2}s'_0 s'_1\cdots$ is divisible by $S$ for some
finite $R'_2$. Moreover, we can choose $R'_2$ to be arbitrary
large (if $R_2$ is infinite, then $R'_2$ can be any divisor of
$R_2$; if $p^\infty|S$, then $R'_2$ can be $p^N$ for any $N\geq
1$). Denote $q_i=\frac{m_0 s'_0\cdots s'_{k_i-1}}{R'_2 p_i}$ and
put $x_i=y_i=q_i$ ($x_i=2q_i$, $y_i=0$ for types $O$ and $C$).
Similar to the proof of (i), the conditions in \refeq{cond4} and
\refeq{cond5} are satisfied, and \refeq{cond6} is equivalent to
the inequality $\frac{m_0 s'_0\cdots s'_{k_i-1}}{R'_2
m_{k_i}}\leq\frac{n_0 s_0\cdots s_{i-1}}{2R_1 n_i}$. Since the
exhaustion of $\ss_1$ is fixed, the right-hand side is bounded by
$\frac{\delta_1}{2R_1}$ from below. But $\frac{m_0 s'_0\cdots
s'_{k_i-1}}{R'_2 m_{k_i}}\leq\frac{1}{R'_2}$, and therefore it is
enough to choose $R'_2$ to be greater than
$\frac{2R_1}{\delta_1}$.

(iv) Choose each $k_i$ large enough so that $m_0 s'_0\cdots
s'_{k_i-1}$ is divisible by $p_i$ and denote $q_i=\frac{m_0
s'_0\cdots s'_{k_i-1}}{p_i}$, $i\geq 0$. Then put $x_i=y_i=q_i$
($x_i=2q_i$, $y_i=0$ for types $O$ and $C$). The conditions in
\refeq{cond4} and \refeq{cond5} are again satisfied, and
\refeq{cond6} is equivalent to the inequality $\frac{m_0
s'_0\cdots s'_{k_i-1}}{m_{k_i}}\leq\frac{n_0 s_0\cdots
s_{i-1}}{2R_1 n_i}$. But $\ss_2$ is sparse, therefore
$\displaystyle\lim_{i\rightarrow\infty}\frac{m_0 s'_0\cdots
s'_i}{m_i}=0$, so the inequality holds for large enough $k_i$.
\end{proof}

\begin{lemma}\label{lemma2}
Let $\ss_1=X_1(\Cal T_1)$ and $\ss_2=X_2(\Cal T_2)$ be diagonal
Lie algebras, neither of them finitary. Set $S_i=\mathrm{Stz}(\Cal
S_i)$, $S=\mathrm{GCD}(S_1,S_2)$, $R_i=\quotst(S_i,S)$, and
$\delta_i=\delta(\Cal T_i)$ for $i=1,2$. We assume that $R_1$ is
finite.
\begin{itemize}
\item[(i)] Assume that $\ss_1$ and $\ss_2$ are non-sparse, both
$R_1$ and $R_2$ are finite, and $S$ is not divisible by an
infinite power of any prime number. In addition, let
$(X_1,X_2)=(A,C)$, $(A,O)$, $(O,C)$, or $(C,O)$. If
$2\frac{R_1}{\delta_1}<\frac{R_2}{\delta_2}$, then $\ss_1$ admits
an injective homomorphism into $\ss_2$. If
$2\frac{R_1}{\delta_1}=\frac{R_2}{\delta_2}$, $\ss_1$ admits an
injective homomorphism into $\ss_2$ unless $\ss_1$ is pure and
$\ss_2$ is dense.

\item[(ii)] Assume that $\ss_1$ and $\ss_2$ are non-sparse, both
$R_1$ and $R_2$ are finite, and $S$ is not divisible by an
infinite power of any prime number. In addition, assume that
$(X_1,X_2)=(C,A)$ or $(O,A)$. If
$\frac{R_1}{\delta_1}<\frac{R_2}{\delta_2}$, then $\ss_1$ admits
an injective homomorphism into $\ss_2$. If
$\frac{R_1}{\delta_1}=\frac{R_2}{\delta_2}$, $\ss_1$ admits an
injective homomorphism into $\ss_2$ unless $\ss_1$ is pure and
$\ss_2$ is dense.

\item[(iii)] Assume that $\ss_1$ and $\ss_2$ are non-sparse. If
$R_2$ is infinite or $S$ is divisible by an infinite power of some
prime number, then $\ss_1$ admits an injective homomorphism into
$\ss_2$.

\item[(iv)] If $\ss_2$ is sparse, then $\ss_1$ admits an injective
homomorphism into $\ss_2$.
\end{itemize}
\end{lemma}
\begin{proof}
The proofs of all four statements in the lemma are analogous to
the corresponding proofs of \refle{lemma1}. We will point out only
the essential differences.

(i) If $(X_1,X_2)=(A,C)$ or $(A,O)$, we put $x_i=y_i=q_i$ as in
the proof of \refle{lemma1} (i). If $(X_1,X_2)=(O,C)$ or $(C,O)$,
we put $x_i=2q_i$, $y_i=0$. Since we are dealing with Lie algebras
of different types we have to pay attention the additional
conditions of \refle{BZhLemma}, which are obviously satisfied. The
rest of the proof is the same and the diagram in \refeq{diag10}
(with Lie algebras of corresponding types) yields an injective
homomorphism of $\ss_1$ into $\ss_2$.

(ii) Since $\ss_1$ is of type $O$ or $C$, $\ss_1$ is one-sided.
The Lie algebra $\ss_2$ is not two-sided symmetric because
$2^\infty$ does not divide $S_2$. Thus $\ss_2$ is either one-sided
or two-sided non-symmetric. Both cases were considered in
\refle{lemma1} (ii) for type $A$ Lie algebras. The construction of
an injective homomorphism of $\ss_1$ into $\ss_2$ is the same in
the case we now consider.

(iii), (iv) If $(X_1,X_2)=(A,C)$ or $(A,O)$, we put $x_i=y_i=q_i$,
and if $(X_1,X_2)=(C,A)$, $(O,A)$, $(O,C)$, or $(C,O)$, we put
$x_i=2q_i$, $y_i=0$. The proofs of (iii) and (iv) are completed in
a similar way to the proofs of \refle{lemma1} (iii) and (iv).
\end{proof}

\begin{corollary}\label{cor1} The three finitary Lie algebras $\sl(\infty)$,
$\so(\infty)$, and $\sp(\infty)$ admit an injective homomorphism
into any diagonal Lie algebra.
\end{corollary}
\begin{proof}
Let $\ss$ be a diagonal Lie algebra. If $\ss$ is finitary, then
$\ss$ is isomorphic to one of the three Lie algebras
$\sl(\infty)$, $\so(\infty)$, $\sp(\infty)$. Hence $\sl(\infty)$,
$\so(\infty)$, admit $\sp(\infty)$ admit an injective homomorphism
into $\ss$. If $\ss$ is not finitary, then (by an easy corollary
from \refle{lemma2} (iii), (iv)) there exists a pure one-sided Lie
algebra of type $A$ $\ss'$ which admits an injective homomorphism
into $\ss$. Then each of the Lie algebras $\sl(\infty)$,
$\so(\infty)$, $\sp(\infty)$ can be mapped by an injective
homomorphism into $\ss'$ by \refprop{prop1}, and the statement
follows.
\end{proof}

\begin{prop}\label{prop2}
Let $\ss_1=X_1(\Cal T_1)$ be a subalgebra of $\ss_2=X_2(\Cal
T_2)$. Set $S_1=\mathrm{Stz}(\Cal S_1)$, $S_2=\mathrm{Stz}(\Cal
S_2)$. Then $S_1|S_2N$ for some $N\in\ZZ_{>0}$.
\end{prop}
\begin{proof} We take $\ss:=\ss_1$ and $\gg:=\ss_2$, in order to
use the notation $\ss_i$ for an exhaustion of $\ss$. Since $\ss$
admits an injective homomorphism into $\gg$ there is a commutative
diagram
$$
\xymatrix{
\ss_1\ar[d]_{\theta_1}\ar[r]&\dots\ar[r]&\ss_i\ar[d]_{\theta_i}\ar[r]&\dots\\
\gg_{k_1}\ar[r]&\dots\ar[r]&\gg_{k_i}\ar[r]&\dots. }
$$
Set $M=I_{\ss_1}^{\gg_{k_1}}(\theta_1)$. Then, by \refprop{Dprop}
(ii), we have $\displaystyle
I_{\gg_{k_1}}^{\gg_{k_i}}M=I_{\ss_1}^{\ss_i}I_{\ss_i}^{\gg_{k_i}}(\theta_i)$
for $i\geq 1$. Then $\displaystyle \prod_{j=1}^{i-1}
I_{\ss_j}^{\ss_{j+1}}|M\prod_{j=k_1}^{k_i-1}
I_{\gg_j}^{\gg_{j+1}}$ for $i\geq 1$. Thus, $S_1|S_2M n_1$, where
$n_1$ is the dimension of the natural representation of $\ss_1$.
\end{proof}

\begin{prop}\label{prop3}
Let $\ss$ be a sparse one-sided Lie algebra of type $A$ not
isomorphic to $\sl(\infty)$. Then $\ss$ admits no non-trivial
homomorphism into a pure one-sided Lie algebra of type $A$.
\end{prop}
\begin{proof}
Assume for the sake of a contradiction that there is an injective
homomorphism of $\ss$ into some pure one-sided Lie algebra of type
$A$. Let $\ss$ be exhausted as
$\sl(n_1)\subset\sl(n_2)\subset\cdots$, each inclusion
$\sl(n_i)\rightarrow\sl(n_{i+1})$ being of signature
$(l_i,0,z_i)$. Recall that by the definition of a sparse Lie
algebra, $\displaystyle \lim_{i\to\infty}\frac{n_1 l_1\cdots
l_{i-1}}{n_i}=0$. Then there is a commutative diagram
\begin{equation}\label{diag8}
\xymatrix{
\sl(n_1)\ar[d]_{\theta_1}\ar[r]&\dots\ar[r]&\sl(n_i)\ar[d]_{\theta_i}\ar[rr]^{(l_i,0,z_i)}&&\sl(n_{i+1})\ar[d]_{\theta_{i+1}}\ar[r]&\dots\\
\sl(m_1)\ar[r]&\dots\ar[r]&\sl(m_1\cdots
m_i)\ar[rr]^{(m_{i+1},0,0)}&&\sl(m_1\cdots m_{i+1})\ar[r]&\dots. }
\end{equation}
The lower row constitutes an exhaustion of the pure Lie algebra
$\sl(m_1 m_2\cdots)$.

Denote by $V_i$ the natural $\sl(m_1\cdots m_i)$-module for $i\geq
1$. Note that $\theta_i$ makes $V_i$ into an $\sl(n_i)$-module.
Let
\begin{equation}\label{eq9}
V_i\downarrow\sl(n_i)\cong\bigoplus_{\lambda\in
H_i}T_\lambda\otimes F_{n_i}^\lambda\end{equation} be the
decomposition into a direct sum of isotypic components. Here
$T_\lambda=\text{Hom}_{\sl(n_i)}(F_{n_i}^\lambda,
V_i\downarrow\sl(n_i))$ is a trivial $\sl(n_i)$-module, and $H_i$
is the set of all highest weights appearing in this decomposition.
We can rewrite \refeq{eq9} (non-canonically) as
\begin{equation}\label{eq10}
V_i\downarrow\sl(n_i)\cong\bigoplus_{\lambda\in
H_i}\underbrace{F_{n_i}^\lambda\oplus\dots\oplus
F_{n_i}^\lambda}_{t_\lambda},\end{equation} where $t_\lambda=\dim
T_\lambda$. Since all weights $\lambda\in H_i$ are dominant, for
each $\lambda=(\lambda_1,\dots,\lambda_{n_i})$,
$\lambda_1-\lambda_{n_i}$ is a non-negative integer. Set
$\displaystyle d_i=\max_{\lambda\in
H_i}(\lambda_1-\lambda_{n_i})$. We define $H(\phi)$ and $d(\phi)$
in a similar way for an arbitrary injective homomorphism $\phi$ of
finite-dimensional classical simple Lie algebras of type $A$, so
that $H(\theta_i)=H_i$ and $d(\theta_i)=d_i$.

Let us show that $d_i\geq d_{i+1}$ for $i\geq 1$. By $\phi_i$ we
denote the injective homomorphism $\displaystyle\sl(m_1\cdots
m_i)\stackrel{(m_{i+1},0,0)}{\longrightarrow}\sl(m_1\cdots
m_{i+1})$ as in \refeq{diag8}. Notice first that
$H(\phi_i\circ\theta_i)=H(\theta_i)=H_i$ and
$\dim\text{Hom}_{\sl(n_i)}(F_{n_i}^\lambda, V_{i+1})=
m_{i+1}\dim\text{Hom}_{\sl(n_i)}(F_{n_i}^\lambda, V_i)$ for all
$\lambda\in H_i$. Furthermore,
$d(\phi_i\circ\theta_i)=d(\theta_i)=d_i$.

Let $\lambda\in H_{i+1}$ be a weight such that
$\lambda_1-\lambda_{n_{i+1}}=d_{i+1}$.  Since $(l_i,0,z_i)$ is the
signature of the diagonal injective homomorphism
$\displaystyle\sl(n_i)\rightarrow\sl(n_{i+1})$, there is a chain
of inclusions $\sl(n_i)\subset\sl(l_i n_i)\subset\sl(l_i
n_i+1)\subset\cdots\subset\sl(l_i n_i+z_i)=\sl(n_{i+1})$ such that
their composition is the original map in \refeq{diag8}. Applying
Gelfand-Tsetlin rule (see \refth{GTrule}) repeatedly we obtain
that $F_{n_{i+1}}^\lambda\downarrow\sl(l_i n_i+z_i-j)$ has a
submodule with highest weight
$(\lambda_1,\lambda_2,\dots,\lambda_{l_i n_i+z_i-j-2},\lambda_{l_i
n_i+z_i-j-1},\lambda_{n_{i+1}})$ for $j=1,\dots,z_i$. We then
apply \refcor{LRcor} to the submodule of
$F_{n_{i+1}}^\lambda\downarrow\sl(l_i n_i)$ with highest weight
$(\lambda_1,\dots,\lambda_{l_i n_i-1},\lambda_{n_{i+1}})$ and see
$\hat{\lambda}:=(\lambda_1+\cdots+\lambda_{l_i},\lambda_{l_i+1}+\cdots+\lambda_{2l_i},\dots,\lambda_{l_i
n_i-l_i+1}+\cdots+\lambda_{l_i n_i-1}+\lambda_{n_{i+1}})\in
H(\phi_i\circ\theta_i)$, i.e. the $\sl(n_i)$-module with highest
weight $\hat{\lambda}$ is a constituent of
$F_{n_{i+1}}^\lambda\downarrow\sl(n_i)$. Hence,
$d(\phi_i\circ\theta_i)\geq(\hat{\lambda}_1-\hat{\lambda}_{n_i})=(\lambda_1+\cdots+\lambda_{l_i})-(\lambda_{l_i
n_i-l_i+1}+\cdots+\lambda_{l_i
n_i-1}+\lambda_{n_{i+1}})\geq\lambda_1-\lambda_{n_{i+1}}=d_{i+1}$,
where the latter inequality holds because $\lambda$ is dominant.
Since $d(\phi_i\circ\theta_i)=d_i$, we have the desired inequality
$d_i\geq d_{i+1}$.

Since $\{d_i\}$ is a decreasing sequence of positive integers, it
stabilizes, so there exists $d\in\ZZ_{>0}$ such that $d_i=d$ for
all $i\geq J$. Pick $K$ such that $l_J\cdots l_{K-1}>d$ (this is
possible since $\ss$ is not isomorphic to $\sl(\infty)$, and
therefore $\displaystyle \prod_{i=1}^\infty l_i$ is infinite).
Consider now the following part of the diagram in \refeq{diag8}:
$$
\xymatrix{
\sl(n_J)\ar[d]_{\theta_J}\ar[r]&\dots\ar[r]&\sl(n_K)\ar[d]_{\theta_K}\\
\sl(m_1\cdots m_J)\ar[r]&\dots\ar[r]&\sl(m_1\cdots m_K).}
$$
The injective homomorphism $\sl(n_J)\rightarrow\sl(n_K)$ is
diagonal of signature $(l,0,z)$, where $l=l_J\cdots l_{K-1}$ and
$z=n_K-ln_J$. Using similar arguments as above we obtain that
$\hat{\lambda}=(\lambda_1+\cdots+\lambda_l,\lambda_{l+1}+\cdots+\lambda_{2l},\dots,\lambda_{n_K-l+1}+\cdots+\lambda_{n_K-1}+\lambda_{n_K})\in
H_J$ for any $\lambda\in H_K$. Then we have
$\lambda_1+\cdots+\lambda_l-(\lambda_{n_K-l+1}+\cdots+\lambda_{n_K})\leq
d$. If $\lambda_{d+1}\neq\lambda_{n_K-d}$, then
$\lambda_{d+1}\geq\lambda_{n_K-d}+1$, in which case
$\lambda_1+\cdots+\lambda_l-(\lambda_{n_K-l+1}+\cdots+\lambda_{n_K})\geq(\lambda_1+\cdots+\lambda_{d+1})-(\lambda_{n_K-d}+\cdots+\lambda_{n_K})\geq
d+1$ as $l>d$. Hence, $\lambda_{d+1}=\lambda_{n_K-d}$ which yields
$\lambda_{d+1}=\lambda_{d+2}=\cdots=\lambda_{n_K-d}$. We thus
conclude that for $i\geq K$ each integral dominant weight
appearing in $H_i$ has the property that all its values apart from
the first $d$ and the last $d$ must be equal.

Let us calculate the index $I_{\sl(n_1)}^{\sl(m_1\cdots m_i)}$ of
the corresponding composition of homomorphisms in \refeq{diag8}.
Using \refprop{Dprop} (ii) and \refcor{Dcor}, we compute
$I_{\sl(n_1)}^{\sl(m_1\cdots m_i)}=I(\theta_1)m_2\cdots m_i$ by
following down $\theta_1$ and to the right; similarly we compute
$I_{\sl(n_1)}^{\sl(m_1\cdots m_i)}=l_1\cdots l_{i-1}I(\theta_i)$
by going to the right and then down $\theta_i$. By \refprop{Dprop}
(iii), (iv) we have
\begin{equation}\label{eq2}
I(\theta_i)=\sum_{\lambda\in H_i}t_\lambda
I(F_{n_i}^\lambda)=\frac{1}{n_i^2-1}\sum_{\lambda\in
H_i}t_\lambda\dim
F_{n_i}^\lambda\langle\lambda,\lambda+2\rho\rangle_{\sl(n_i)},
\end{equation} where $2\rho$ is the sum of all the positive roots of
$\sl(n_i)$.

Note that
$\langle\lambda,\lambda+2\rho\rangle_{\sl(n_i)}=(\tilde{\lambda},\tilde{\lambda}+2\rho)$,
where
$\displaystyle\tilde{\lambda}_j=\lambda_j-\frac{1}{n_i}\sum_{k=1}^{n_i}\lambda_k$
for $j=1,\dots,n_i$, $2\rho=(n_i-1,n_i-3,\dots,-(n_i-1))$, and
$(\;,\;)$ is the usual scalar product on $\CC^{n_i}$.

Fix $i\geq K$, using the notation from above, so that
$\lambda_1-\lambda_{n_i}\leq d$ and
$\lambda_{d+1}=\lambda_{d+1}=\cdots=\lambda_{n_i-d}$. Set
$\alpha=\tilde{\lambda}_{d+1}$, so that
$|\tilde{\lambda}_j-\alpha|=0$ for $j=d+1,d+2,\dots,n_i-d$. Then
$|\tilde{\lambda}_j-\alpha|=|\lambda_j-\lambda_{d+1}|\leq d$ for
all $j$. Since $\displaystyle\sum_{j=1}^{n_i}\tilde{\lambda}_j=0$
and
$\tilde{\lambda}_1-\tilde{\lambda}_{n_i}=\lambda_1-\lambda_{n_i}\leq
d$, we have $|\tilde{\lambda}_j|\leq d$ for all $j$. Hence,
\begin{align*}
|\langle\lambda,\lambda+2\rho\rangle_{\sl(n_i)}|
&=|(\tilde{\lambda},\tilde{\lambda}+2\rho)|
=\left|\sum_{j=1}^{n_i}\tilde{\lambda}_j(\tilde{\lambda}_j+n_i-2j+1)\right|
\\
&=\left|\sum_{j=1}^{n_i}\tilde{\lambda}_j(\tilde{\lambda}_j-\alpha-2j)
+(n_i+1+\alpha)\sum_{j=1}^{n_i}\tilde{\lambda}_j\right| \\
&=\left|\sum_{j=1}^{n_i}\tilde{(\lambda}_j-\alpha+\alpha)(\tilde{\lambda}_j-\alpha-2j)\right|\\
&=\left|\sum_{j=1}^{n_i}(\tilde{\lambda}_j
-\alpha)^2-2\sum_{j=1}^{n_i}(\tilde{\lambda}_j-\alpha)j
+\sum_{i=1}^{n_i}(\alpha(\tilde{\lambda}_j-\alpha)-2\alpha
j)\right|\\ &=\left|\sum_{j=1}^{n_i}(\tilde{\lambda}_j-\alpha)^2
-2\sum_{j=1}^d(\tilde{\lambda}_j-\alpha)j-2\sum_{j=n_i-d+1}^{n_i}(\tilde{\lambda}_j
-\alpha)j-n_i\alpha^2-n_i(n_i+1)\alpha\right|\\
&\leq\sum_{j=1}^{n_i}d^2+2\sum_{j=1}^d j d
+2\sum_{j=n_i-d+1}^{n_i}j d+n_i\alpha^2+n_i(n_i+1)|\alpha|\\
&=2n_i d^2+2(n_i+1)d^2+n_i\alpha^2+n_i(n_i+1)|\alpha|.
\end{align*} Since
$\tilde{\lambda}_1+\cdots+\tilde{\lambda}_d+\alpha(n_i-2d)+\tilde{\lambda}_{n_i-d+1}+\cdots+\tilde{\lambda}_{n_i}=0$
(which implies $|\alpha|\leq\frac{2d^2}{n_i-2d}$), we obtain the
following inequality:
$$|\langle\lambda,\lambda+2\rho\rangle_{\sl(n_i)}|\leq 2d^2
n_i+2d^2(n_i+1)+\frac{4d^4
n_i}{(n_i-2d)^2}+\frac{2d^2n_i(n_i+1)}{n_i-2d}\leq c_0 n_i$$ for
all $i\geq K$, where $c_0$ is some positive constant. Then from
\refeq{eq2} we have $\displaystyle I(\theta_i)\leq\frac{c_0
n_i}{n_i^2-1}\sum_{\lambda\in H_i}t_\lambda\dim
F_{n_i}^\lambda=\frac{c_0 n_i}{n_i^2-1}m_1\cdots m_i$. Hence,
$I(\theta_1)m_2\cdots m_i=I_{\sl(n_1)}^{\sl(m_1\cdots
m_i)}=l_1\cdots l_{i-1}I(\theta_i)\leq l_1\cdots l_{i-1}\frac{c_0
n_i}{n_i^2-1}m_1\cdots m_i$. This implies $\frac{I(\theta_1)}{c_0
m_1}\leq l_1\cdots l_{i-1}\frac{n_i}{n_i^2-1}$, so
$\frac{l_1\cdots l_{i-1}}{n_i}\geq c_1$ for some positive constant
$c_1$. The last inequality contradicts the fact that
$\displaystyle \lim_{i\rightarrow\infty}\frac{n_1 l_1\cdots
l_{i-1}}{n_i}=0$, so the proposition follows.
\end{proof}

\begin{corollary}\label{cor2}
Let $\ss_1$, $\ss_2$ be non-finitary diagonal Lie algebras. Assume
that $\ss_1$ is sparse and there is an injective homomorphism of
$\ss_1$ into $\ss_2$. Then $\ss_2$ must be sparse as well.
\end{corollary}
\begin{proof}
Suppose, on the contrary, that $\ss_2$ is pure or dense.
\refle{lemma2} (iv) implies that there exists a sparse one-sided
Lie algebra of type $A$ $\ss'_1$ which admits an injective
homomorphism into $\ss_1$. By \refle{lemma2} (iii) there exists a
pure one-sided Lie algebra of type $A$ $\ss'_2$ such that $\ss_2$
admits an injective homomorphism into $\ss'_2$. If $\ss_1$ would
admit an injective homomorphism into $\ss_2$, then $\ss'_1$ would
admit an injective homomorphism into $\ss'_2$ through the chain
$\ss'_1\subset\ss_1\subset\ss_2\subset\ss'_2$, which would
contradict \refprop{prop3}. Hence the statement holds.
\end{proof}

\begin{prop}\label{prop4}
Let $\ss_1=A(\Cal T_1)$ and $\ss_2=A(\Cal T_2)$ be pure one-sided
Lie algebras, neither of them finitary. Set $S_i=\mathrm{Stz}(\Cal
S_i)$ for $i=1,2$, and $S=\mathrm{GCD}(S_1,S_2)$. Assume that both
Steinitz numbers $\quotst(S_1,S)$ and $\quotst(S_2,S)$ are finite
and $S$ is not divisible by an infinite power of any prime number.
An injective homomorphism of $\ss_1$ into $\ss_2$ is necessarily
diagonal.
\end{prop}
\begin{proof}
Let $S=p_1^{l_1} p_2^{l_2}\cdots$ for the increasing sequence
$\{p_i\}$ of all prime numbers dividing $S$. Denote
$n_i=\frac{S_1}{S}(p_1)^{l_1}\cdots (p_{N+i})^{l_{N+i}}$ for
$i\geq 0$, with integer $N$ to be fixed later. Suppose that there
is an injective homomorphism $\theta:\ss_1\to\ss_2$. Then it is
given by the following commutative diagram:
\begin{equation}\label{diag5}
\xymatrix{
\sl(n_0)\ar[d]_{\theta_0}\ar[r]&\dots\ar[r]&\sl(n_i)\ar[d]_{\theta_i}\ar[r]&\sl(n_{i+1})\ar[d]_{\theta_{i+1}}\ar[r]&\dots\\
\sl(m_0)\ar[r]&\dots\ar[r]&\sl(m_i)\ar[r]&\sl(m_{i+1})\ar[r]&\dots,
}
\end{equation} where $m_i=\frac{S_2}{S} (p_1)^{l_1}\cdots
(p_{N+k_i})^{l_{N+k_i}}$ for $i\geq 0$ for some $k_0,k_1,\dots$.
By possibly shifting the bottom row of the diagram we may assume
that $k_i\geq i+1$ for each $i\geq 0$.

Denote by $W_i$ the natural $\sl(m_i)$-module. Let $H(\phi)$ and
$d(\phi)$ be as in the proof of \refprop{prop2} for an arbitrary
injective homomorphism $\phi$ of finite-dimensional classical
simple Lie algebras of type $A$. Set $H_i=H(\theta_i)$ and
$d_i=d(\theta_i)$ for $i\geq 0$. Similarly to \refeq{eq10} we then
have
$$
W_i\downarrow\sl(n_i)\cong\bigoplus_{\lambda\in
H_i}\underbrace{F_{n_i}^\lambda\oplus\dots\oplus
F_{n_i}^\lambda}_{t_{\lambda,i}},$$ where
$t_{\lambda,i}=\dim\text{Hom}_{\sl(n_i)}(F_{n_i}^\lambda,
W_i\downarrow\sl(n_i))$.

Similarly to the proof of \refprop{prop2}, $\{d_i\}$ is a
decreasing sequence, and therefore $d_i=d$ for $i\geq i_0$. By
choosing $N$ large enough we make $d_i=d$ and $p_{N+i}> d+1$ for
all $i\geq 0$. Take now $0\leq i<j\leq k_i$ and consider the
following piece of the diagram in \refeq{diag5}:
\begin{equation}\label{diag6}
\xymatrix{
\sl(n_i)\ar[d]_{\theta_i}\ar[r]&\dots\ar[r]&\sl(n_j)\ar[d]_{\theta_j}\\
\sl(m_i)\ar[r]&\dots\ar[r]&\sl(m_j).}
\end{equation}
Here the injective homomorphism $\sl(n_i)\rightarrow\sl(n_j)$ is
of signature $(q,0,0)$, where $q=(p_{N+i+1})^{l_{N+i+1}}\cdots
(p_{N+j})^{l_{N+j}}$. Take an arbitrary non-trivial highest weight
$\lambda$ in $H_j$, yielding the $\sl(n_j)$-module
$F_{n_j}^\lambda$. Since $n_j=qn_i$, by \refprop{LRprop} we have
$$
F_{qn_i}^\lambda\downarrow\sl(n_i)\cong\bigoplus_\nu
(\sum_{\mu_1,\dots,\mu_q}c_{\mu_1\dots\mu_q}^\lambda
c_{\mu_1\dots\mu_q}^\nu)F_{n_i}^\nu.$$ Since the coefficients
$c_{\mu_1\dots\mu_q}^\lambda$ and $c_{\mu_1\dots\mu_q}^\nu$ are
independent of the order of $\mu_1,\dots,\mu_q$, we can rewrite
this as
\begin{equation}\label{eq7}
F_{qn_i}^\lambda\downarrow\sl(n_i)\cong\bigoplus_\nu
(\sum_{[\mu_1,\dots,\mu_q]}C_q^{q_1,\dots,q_r}c_{\mu_1\dots\mu_q}^\lambda
c_{\mu_1\dots\mu_q}^\nu)F_{n_i}^\nu.\end{equation} Here
$[\mu_1,\dots,\mu_q]$ denotes the multiset with these elements,
and $q_1,\dots,q_r$ are the corresponding multiplicities, so that
$q_1+\cdots+q_r=q$.

Fix a highest weight $\nu$ such that $F_{n_i}^\nu$ has non-zero
multiplicity in \refeq{eq7} and fix a multiset of integral
dominant weights $[\mu_1,\dots,\mu_q]$ making both generalized
Littlewood-Richardson coefficients $c_{\mu_1\dots\mu_q}^\lambda$
and $c_{\mu_1\dots\mu_q}^\nu$ non-zero. We will show that $q$
divides $C_q^{q_1,\dots,q_r}$ (and hence the contribution from
$[\mu_1,\dots,\mu_q]$ to the multiplicity of $F_{n_i}^\nu$) if the
module $F_{n_i}^\nu$ is non-trivial. Suppose that $p_l$ divides
all $q_1,\dots,q_r$ for some $N+i+1\leq l\leq N+j$. Note that the
$\sl(n_i)$-module $F_{n_i}^{\nu'}$ for $\nu'=\mu_1+\cdots+\mu_q$
also has non-zero multiplicity in \refeq{eq7} because
$c_{\mu_1\dots\mu_q}^{\nu'}\neq 0$. Since all $q_1,\dots,q_r$ are
divisible by $p_l$, we have $\nu'=p_l \mu'$ for some integral
dominant weight $\mu'$. Since $F_{n_i}^{\nu'}$ has non-zero
multiplicity in $W_j$ considered as an $\sl(n_i)$-module using the
path along $\theta_j$ in \refeq{diag6}, and since
$W_j\downarrow\sl(m_i)$ is a direct sum of copies of $W_i$, it
must be that $F_{n_i}^{\nu'}$ has non-zero multiplicity in
$W_i\downarrow\sl(n_i)$, i.e. $\nu'\in H_i$. Since $d_i=d<p_l-1$
we have $p_l>\nu'_1-\nu'_{n_i}=p_l(\mu'_1-\mu'_{n_i})$ which
possible only if $\mu'_1=\mu'_{n_i}$ (equivalently,
$\nu'_1=\nu'_{n_i}$). Therefore $\nu'$ is a trivial highest
weight, and hence all $\mu_1,\dots,\mu_q$ are trivial as well.
Then the coefficient $c_{\mu_1\dots\mu_q}^\nu$ is non-zero only if
$\nu$ is trivial, so $F_{n_i}^\nu$ is the trivial module.

Suppose now that $p_l$ does not divide at least one of
$q_1,\dots,q_r$ for each $l$ such that $N+i\leq l\leq N+j$. A
combinatorial argument shows that
$C_q^{q_1,\dots,q_r}=\frac{q!}{q_1!\cdots q_r!}$ is divisible by
$q$ if each prime divisor of $q$ fails to divide at least one of
$q_1,\dots,q_r$. We thus conclude that each non-trivial
$\sl(n_i)$-module $F_{n_i}^\nu$ with non-zero multiplicity in
\refeq{eq7}, has multiplicity divisible by $q$. As a corollary,
any non-trivial simple constituent of $W_j\downarrow\sl(n_i)$
appears with multiplicity divisible by $q$.

By following the diagram in \refeq{diag6} down $\theta_i$ and then
to the right, we get $\displaystyle W_j\downarrow\sl(n_i)\cong
\frac{m_j}{m_i}\bigoplus_{\nu\in H_i} t_{\nu,i} F_{n_i}^\nu$.
Since $q=(p_{N+i+1})^{l_{N+i+1}}\cdots (p_{N+j})^{l_{N+j}}$ is
relatively prime to
$\frac{m_j}{m_i}=(p_{N+k_i+1})^{l_{N+k_i+1}}\cdots
(p_{N+k_j})^{l_{N+k_j}}$ (as $j\leq k_i$), the commutativity of
the diagram in \refeq{diag6} implies that $t_{\nu,i}$ is divisible
by $q$ for any non-trivial $\nu$ in $H_i$.

Let us introduce a new notation. For an arbitrary injective
homomorphism $\phi:\gg_1\rightarrow\gg_2$ of finite-dimensional
classical simple Lie algebras of type $A$ we denote by $N(\phi)$
the number (counting multiplicities) of simple non-trivial
constituents of the natural representation of $\gg_2$ considered
as a $\gg_1$-module via $\phi$. Then $\displaystyle
N_i:=N(\theta_i)$ is divisible by $q=(p_{N+i+1})^{l_{N+i+1}}\cdots
(p_{N+j})^{l_{N+j}}$ by the above argument. Taking $j=k_i$ we
obtain that $N_i$ is divisible by $(p_{N+i+1})^{l_{N+i+1}}\cdots
(p_{N+k_i})^{l_{N+k_i}}$.

Fix now $j=i+1$ in the diagram in \refeq{diag6}, and let
$\psi:\sl(n_i)\rightarrow\sl(m_{i+1})$ denote the map produced by
this diagram. As shown above, each non-trivial weight $\lambda\in
H_{i+1}$ yields a non-trivial weight in $H(\psi)=H_i$ with
non-zero multiplicity divisible by $(p_{N+i+1})^{l_{N+i+1}}$, and
hence at least $(p_{N+i+1})^{l_{N+i+1}}$. Therefore by following
the diagram to the right and then down $\theta_{i+1}$, we obtain
$N(\psi)\geq (p_{N+i+1})^{l_{N+i+1}}N_{i+1}$. Note also that
equality holds here only if for each non-trivial $\lambda\in
H_{i+1}$ we have $F_{qn_i}^\lambda\downarrow\sl(n_i)\cong q
F_{n_i}^\nu\oplus T$ for a non-trivial $\nu\in H_i$, where $T$ is
a trivial (possibly 0-dimensional) module. Meanwhile, by following
the diagram down $\theta_i$ and to the right we have
$N(\psi)=(p_{N+k_i+1})^{l_{N+k_i+1}}\cdots
(p_{N+k_{i+1}})^{l_{N+k_{i+1}}}N_i$. As a result we obtain the
inequality $(p_{N+k_i+1})^{l_{N+k_i+1}}\cdots
(p_{N+k_{i+1}})^{l_{N+k_{i+1}}}N_i\geq
(p_{N+i+1})^{l_{N+i+1}}N_{i+1}$, i.e. $\alpha_i\geq \alpha_{i+1}$,
where $\alpha_i:=\frac{N_i}{(p_{N+i+1})^{l_{N+i+1}}\cdots
(p_{N+k_i})^{l_{N+k_i}}}$ are integers for $i\geq 0$. Since
$\{\alpha_i\}$ is a decreasing sequence of positive integers it
stabilizes, and by choosing $N$ sufficiently large we can assume
that $\alpha_0=\alpha_1=\alpha_2=\cdots$.

Now take an arbitrary non-trivial $\lambda\in H_{i+1}$. Since
$\alpha_i=\alpha_{i+1}$, the decomposition in \refeq{eq7} becomes
$F_{qn_i}^\lambda\downarrow\sl(n_i)\cong q F_{n_i}^\nu\oplus T$
for some non-trivial $\nu\in H_i$, where $T$ is some trivial
(possibly 0-dimensional) module. Since the contribution from each
multiset $[\mu_1,\dots,\mu_q]$ to the multiplicity of
$F_{n_i}^\nu$ in \refeq{eq7} is divisible by $q$, there exists
exactly one multiset $[\mu_1,\dots,\mu_q]$ making a non-zero
contribution to the multiplicity of $F_{n_i}^\nu$. Moreover, the
fact that $C_q^{q_1,\dots,q_r}c_{\mu_1\dots\mu_q}^\lambda
c_{\mu_1\dots\mu_q}^\nu=q$ together with the fact that $q$ divides
$C_q^{q_1,\dots,q_r}$ implies $C_q^{q_1,\dots,q_r}=q$. It is easy
to check that $\frac{q!}{q_1!\cdots q_r!}=q$ only if $r=2$ and
$\{q_1,q_2\}=\{1,q-1\}$. Then we safely can assume that
$\mu_1=\mu_2=\cdots=\mu_{q-1}$. Since $\nu'=\mu_1+\cdots+\mu_q$ is
a non-trivial weight satisfying $c_{\mu_1\dots\mu_q}^{\nu'}\neq
0$, the module $F_{n_i}^{\nu'}$ also has non-zero multiplicity in
\refeq{eq7}, and therefore $\nu=\nu'$. Hence
$\nu=(q-1)\mu_1+\mu_q$, and since $\nu_1-\nu_{n_i}\leq
d<(p_{N+i+1})^{l_{N+i+1}}-1=q-1$, we immediately get that $\mu_1$
is a trivial weight. Then the only multiset $[\mu_1,\dots,\mu_q]$
making $c_{\mu_1\dots\mu_q}^\lambda$ non-zero has $q-1$ trivial
weights. One can check that this is only possible if $\lambda$ is
either of the form $(c+1,c,\dots,c,c)$ or $(c,c,\dots,c,c+1)$.
Thus, all non-trivial highest weights from $H_{i+1}$ are either
those of the natural or of the conatural representation. This
means precisely that all homomorphisms $\theta_i$ are diagonal.
\end{proof}
\begin{corollary}\label{cor3}
Let $\ss_1=X_1(\Cal T_1)$ and $\ss_2=X_2(\Cal T_2)$ be non-sparse
Lie algebras, neither of them finitary. Set $S_i=\mathrm{Stz}(\Cal
S_i)$, $S=\mathrm{GCD}(S_1,S_2)$, and $R_i=\quotst(S_i,S)$ for
$i=1,2$. Assume that $S$ is not divisible by an infinite power of
any prime number, and that both $R_1$ and $R_2$ are finite. An
injective homomorphism of $\ss_1$ into $\ss_2$ is necessary
diagonal.
\end{corollary}
\begin{proof}
Set $\delta_i=\delta(\Cal T_i)$, $i=1,2$. Denote
$\ss'_1=\sl(\quotst(S_1,R'_1))$, where $R'_1>2\delta_1$ is some
finite divisor of $S_1$, and $\ss'_2=\sl(S_2R'_2)$, where $R'_2$
is finite and $R'_2>\frac{2}{\delta_2}$. Then, by \refle{lemma1}
(i) and \refle{lemma2} (i), (ii), $\ss'_1$ admits an injective
homomorphism into $\ss_1$ and $\ss_2$ admits an injective
homomorphism into $\ss'_2$. Then there exists an injective
homomorphism of $\ss'_1$ into $\ss'_2$ through the chain
$\ss'_1\subset\ss_1\subset\ss_2\subset\ss'_2$ and this
homomorphism is diagonal because the Lie algebras $\ss'_1$ and
$\ss'_2$ satisfy the conditions of \refprop{prop4}. Finally, it
follows from \refcor{diagCor} that the injective homomorphism of
$\ss_1$ into $\ss_2$ has to be diagonal as well.
\end{proof}

\begin{lemma}\label{lemma3}
Let $\ss_1=X_1(\Cal T_1)$ and $\ss_2=X(\Cal T_2)$ be non-sparse
Lie algebras, neither of them finitary. Set $S_i=\mathrm{Stz}(\Cal
S_i)$, $S=\mathrm{GCD}(S_1,S_2)$, $R_i=\quotst(S_i,S)$,
$\delta_i=\delta(\Cal T_i)$, $C_i=\mathrm{Stz}(\Cal C_i)$,
$C=\mathrm{GCD}(C_1,C_2)$, $B_i=\quotst(C_i,C)$, and
$\sigma_i=\sigma(\Cal T_i)$ for $i=1,2$. Assume that $S$ is not
divisible by an infinite power of any prime, and both $R_1$ and
$R_2$ are finite. If $\ss_1$ admits a diagonal injective
homomorphism into $\ss_2$, then the following holds.
\begin{itemize}
\item[(i)] $\frac{R_1}{\delta_1}\leq\frac{R_2}{\delta_2}$. The
inequality is strict if $\ss_1$ is pure and $\ss_2$ is dense.

\item[(ii)] $2\frac{R_1}{\delta_1}\leq\frac{R_2}{\delta_2}$ when
one of the following additional hypotheses holds:
\begin{itemize}
\item $(X_1,X_2)=(A,C)$, $(A,O)$, $(O,C)$, or $(C,O)$; \item
$(X_1,X_2)=(A,A)$, $B_1$ is infinite; \item $(X_1,X_2)=(A,A)$,
$B_1$ is finite, $\ss_1$ is two-sided weakly non-symmetric,
$\ss_2$ is either one-sided or two-sided strongly non-symmetric;
\item $(X_1,X_2)=(A,A)$, both $B_1$ and $B_2$ are finite, $C$ is
not divisible by an infinite power of a prime number, both
$\ss_1$, $\ss_2$ are two-sided strongly non-symmetric, and
$\frac{R_1\sigma_1}{B_1}<\frac{R_2\sigma_2}{B_2}$.
\end{itemize}
Again the inequality is strict if $\ss_1$ is pure and $\ss_2$ is
dense.
\end{itemize}
\end{lemma}
\begin{proof}
As it was explained in the proof of \refle{lemma1}, we can choose
suitable exhaustions of $\ss_1$ and $\ss_2$.

(i) Assume that $(X_1,X_2)=(A,A)$ (the other cases are analogous).
Let $\ss_1$ be exhausted as
$\sl(n_0)\subset\sl(n_1)\subset\cdots$, each inclusion
$\sl(n_i)\rightarrow\sl(n_{i+1})$ being of signature
$(l_i,r_i,z_i)$, $i\geq 0$ and $\ss_2$ as
$\sl(m_0)\subset\sl(m_1)\subset\cdots$ with
$\sl(m_i)\rightarrow\sl(m_{i+1})$ being of signature
$(l'_i,r'_i,z'_i)$, $i\geq 0$. Moreover, we choose $n_0$ to be
divisible by $R_1$ and $m_0$ to be divisible by $R_2$.

There is a commutative diagram
\begin{equation}\label{diag9}
\xymatrix{
\sl(n_0)\ar[d]_{\theta_0}\ar[r]&\sl(n_1)\ar[d]_{\theta_1}\ar[r]&\dots\ar[r]&\sl(n_i)\ar[d]_{\theta_i}\ar[r]&\dots\\
\sl(m_{k_0})\ar[r]&\sl(m_{k_1})\ar[r]&\dots\ar[r]&\sl(m_{k_i})\ar[r]&\dots,}
\end{equation}
where each injective homomorphism $\theta_i$ is diagonal of
signature $(x_i,y_i,m_{k_i}-(x_i+y_i)n_i)$. Denote $q_i=x_i+y_i$.
Then, using Corollary 2.5 \cite{BZ}, we get
\begin{equation}\label{eq3}q_i s'_{k_i}\cdots s'_{k_j-1}=s_i\cdots
s_{j-1} q_j\text{ for all } j>i\geq 0.\end{equation} Hence $s_i
s_{i+1}\cdots$ divides $q_is'_{k_i}s'_{k_i+1}\cdots$ for $i\geq
0$, so $S_1 m_0 s'_0\cdots s'_{k_i-1}$ divides $q_i S_2 n_0
s_0\cdots s_{i-1}$. Since $S$ is not divisible by an infinite
power of any prime number, the first Steinitz number will still
divide the second one after cancellation of both of them by $S$.
Therefore $\quotst(q_i R_2 n_0 s_0\cdots s_{i-1},R_1 m_0
s'_0\cdots s'_{k_i-1})$ is a Steinitz number which is moreover
finite, and thus it is a positive integer. So $\frac{m_0
s'_0\cdots s'_{k_i-1}}{R_2 m_{k_i}}\leq\frac{n_0 s_0\cdots
s_{i-1}}{R_1 n_i}$. Taking the limit of both sides for
$i\rightarrow\infty$ we get
$\frac{\delta_2}{R_2}\leq\frac{\delta_1}{R_1}$. Moreover, if
$\ss_1$ is pure and $\ss_2$ is dense, then $\frac{m_0 s'_0\cdots
s'_{k_i-1}}{R_2 m_{k_i}}\leq\frac{\delta_1}{R_1}$ for large enough
$i$. But the decreasing sequence $\frac{m_0 s'_0\cdots
s'_{k_i-1}}{m_{k_i}}$ does not stabilize, so we obtain the strict
inequality $\frac{\delta_2}{R_2}<\frac{\delta_1}{R_1}$.

(ii) We keep the notations from (i). The injective homomorphism of
$\ss_1$ into $\ss_2$ is given again by \refeq{diag9}. If the pair
$(X_1,X_2)$ is one of $(A,C)$, $(A,O)$, $(O,C)$, and $(C,O)$,
then, by Proposition 2.3 \cite{BZ}, for any diagonal injective
homomorphism of a type $X_1$ algebra into a type $X_2$ algebra of
signature $(l,r,z)$ the integer $l+r$ is even. Therefore $q_j$ is
divisible by 2 for any $j$ and it follows from \refeq{eq3} that
$q_is'_{k_i}s'_{k_i+1}\cdots$ is divisible by $2s_i
s_{i+1}\cdots$. The rest of the proof is analogous to (i).

In the other three cases both $\ss_1$ and $\ss_2$ are of type $A$.
Notice that neither $\ss_1$ nor $\ss_2$ is two-sided symmetric
(otherwise $S$ would be divisible by $2^\infty$). Thus we can
assume that $c_i>0$ and $c'_i>0$ for all $i\geq 0$. Denote
$t_i=x_i-y_i$. It is enough to prove that $t_i=0$ for infinitely
many $i$ ( because then $q_i$ is even for infinitely many $i$ and
the statement can be proven similarly to the first case). Assume
the contrary, i.e. let $t_i>0$ for $i\geq i_0$. Without loss of
generality we can assume that $t_i>0$ for all $i\geq 0$. Let us
show that this contradicts with the assumptions of the lemma in
all three cases.

Let $B_1$ be infinite. By Corollary 2.5 in \cite{BZ},
\begin{equation}\label{eq4}t_0 c'_{k_0}\cdots c'_{k_i-1}=c_0\cdots
c_{i-1}t_i\text{ for }i\geq 1.\end{equation} Then clearly $c_0
c_1\cdots$ divides $t_0 c'_{k_0}c'_{k_0+1}\cdots$, and therefore
$B_1$ divides $n_0 t_0$. This contradicts $B_1$ being infinite.

For the next case, combining \refeq{eq3} and \refeq{eq4}, we
obtain $\frac{t_0}{q_0}\cdot\frac{c'_{k_0}\cdots
c'_{k_i-1}}{s'_{k_0}\cdots
s'_{k_i-1}}=\frac{t_i}{q_i}\cdot\frac{c_0\cdots c_{i-1}}{s_0\cdots
s_{i-1}}$. By definition
$\displaystyle\sigma_1=\lim_{i\rightarrow\infty}\frac{c_0\cdots
c_i}{s_0\cdots s_i}$, and since $\ss_1$ is two-sided weakly
non-symmetric we have
$\displaystyle\lim_{i\rightarrow\infty}\frac{t_i}{q_i}\frac{c_0\cdots
c_i}{s_0\cdots s_i}=0$. But
$\displaystyle\lim_{i\rightarrow\infty}\frac{t_0}{q_0}\cdot\frac{c'_{k_0}\cdots
c'_{k_i-1}}{s'_{k_0}\cdots s'_{k_i-1}}=u\sigma_2$, where
$u=\frac{t_0 s'_0\cdots s'_{k_0-1}}{q_0 c'_0\cdots c'_{k_0-1}}>0$.
So $\sigma_2=0$, contradicting $\ss_2$ being not two-sided weakly
non-symmetric.

Finally, let both $\ss_1$ and $\ss_2$ be two-sided strongly
non-symmetric. Since $t_i\leq q_i$ for $i\geq 0$, we have
$\frac{t_0}{q_0}\cdot\frac{c'_{k_0}\cdots
c'_{k_i-1}}{s'_{k_0}\cdots s'_{k_i-1}}\leq\frac{c_0\cdots
c_{i-1}}{s_0\cdots s_{i-1}}$. Taking the limit we obtain
\begin{equation}\label{eq5}\frac{t_0}{q_0}\cdot\frac{s'_0\cdots
s'_{k_0-1}}{c'_0\cdots
c'_{k_0-1}}\sigma_2\leq\sigma_1.\end{equation} Let us go back to
\refeq{eq3}. We know that $q_0 s'_{k_0}\cdots s'_{k_i-1}=s_0\cdots
s_{i-1}q_i$. If $q_i$ is divisible by some prime number $p$ for
infinitely many $i$, then by an argument similar to that in (i)
one derives the inequality
$p\frac{R_1}{\delta_1}\leq\frac{R_2}{\delta_2}$, from which the
statement follows. So we can assume that every $p$ divides at most
finitely many $q_i$. Then it is easy to see that the Steinitz
numbers $q_0 s'_{k_0}s'_{k_0+1}\cdots$ and $s_0 s_1\cdots$ have
equal values at every prime $p$, so they coincide. Hence,
\begin{equation}\label{eq6} \frac{R_2}{R_1}=\frac{m_0 s'_0\cdots
s'_{k_0-1}}{q_0 n_0}.\end{equation} From \refeq{eq4} $c_0
c_1\cdots$ divides $t_0 c'_{k_0}c'_{k_0+1}\cdots$, and therefore
$\frac{B_2}{B_1}\geq\frac{m_0 c'_0\cdots c'_{k_0-1}}{t_0 n_0}$.
Combining the latter inequality with \refeq{eq5} and \refeq{eq6}
we obtain $\frac{\sigma_1}{\sigma_2}\geq\frac{R_2 B_1}{R_1 B_2}$,
which contradicts an assumption in the statement of the lemma.
\end{proof}

We are now able to prove the main result of the paper.
\begin{theo}\label{Theorem}
\begin{itemize}
\item[a)] The three finitary Lie algebras $\sl(\infty)$,
$\so(\infty)$, $\sp(\infty)$ admit an injective homomorphism into
any infinite-dimensional diagonal Lie algebra. An
infinite-dimensional non-finitary diagonal Lie algebra admits no
injective homomorphism into $\sl(\infty)$, $\so(\infty)$,
$\sp(\infty)$.

\item[b)] Let $\ss_1=X_1(\Cal T_1)$, $\ss_2=X_2(\Cal T_2)$ be
infinite-dimensional non-finitary diagonal Lie algebras. Set
$S_i=\mathrm{Stz}(\Cal S_i)$, $S=\mathrm{GCD}(S_1,S_2)$,
$R_i=\quotst(S_i,S)$, $\delta_i=\delta(\Cal T_i)$,
$C_i=\mathrm{Stz}(\Cal C_i)$, $C=\mathrm{GCD}(C_1,C_2)$,
$B_i=\quotst(C_i,C)$, and $\sigma_i=\sigma(\Cal T_i)$ for $i=1,2$.
Then $\ss_1$ admits an injective homomorphism into $\ss_2$ if and
only if the following conditions hold.

\begin{itemize}
\item[1)] $R_1$ is finite.

\item[2)] $\ss_2$ is sparse if $\ss_1$ is sparse.

\item[3)] If $\ss_1$ and $\ss_2$ are non-sparse, both $R_1$ and
$R_2$ are finite, and $S$ is not divisible by an infinite power of
any prime number, then
$\epsilon\frac{R_1}{\delta_1}\leq\frac{R_2}{\delta_2}$ for
$\epsilon$ as specified below. The inequality is strict if $\ss_1$
is pure and $\ss_2$ is dense. We have $\epsilon=2$, except in the
cases listed below, in which $\epsilon = 1$:
\begin{itemize}
\item[3.1)] $(X_1,X_2)=(C,C)$, $(O,O)$, $(C,A)$, $(O,A)$, and
$(X_1,X_2)=(A,A)$ with both $\ss_1$ and $\ss_2$ being one-sided;

\item[3.2)] $(X_1,X_2)=(A,A)$, $B_1$ is finite, either $\ss_1$ is
one-sided and $\ss_2$ is two-sided non-symmetric or $\ss_2$ is
two-sided weakly non-symmetric and $\ss_1$ is two-sided
non-symmetric;

\item[3.3)] $(X_1,X_2)=(A,A)$, $B_1$ is finite, both $\ss_1$ and
$\ss_2$ are two-sided strongly non-symmetric, either $B_2$ is
infinite or $C$ is divisible by an infinite power of any prime
number;

\item[3.4)] $(X_1,X_2)=(A,A)$, both $B_1$ and $B_2$ are finite,
both $\ss_1$ and $\ss_2$ are two-sided strongly non-symmetric, $C$
is not divisible by an infinite power of any prime number, and
$\frac{R_1\sigma_1}{B_1}\geq\frac{R_2\sigma_2}{B_2}$.
\end{itemize}
\end{itemize}
\end{itemize}
\end{theo}
\begin{proof}
a) The statement follows directly from \refcor{cor1} and
\refprop{prop2}.

b) The sufficiency of the conditions follows directly from
\refle{lemma1} and \refle{lemma2}.

The necessity of conditions 1 and 2 follows from \refprop{prop2}
and \refcor{cor2} respectively. Let us prove the necessity of
condition 3. Note that the assumptions of this condition satisfy
\refcor{cor3}. Hence in this case an injective homomorphism of
$\ss_1$ into $\ss_2$, if it exists, has to be diagonal. Therefore
we can apply \refle{lemma3} and this lemma implies the necessity
of condition 3 (it is easy to check that under corresponding
assumptions the cases which are not listed in 3.1$-$3.4 are
exactly the cases listed in \refle{lemma3} (ii)).
\end{proof}

The following corollary gives a description of equivalence classes
of diagonal Lie algebras with respect to the equivalence relation
introduced earlier in this paper.
\begin{corollary}\label{Corollary1}
\begin{itemize}
\item[a)] The three finitary Lie algebras $\sl(\infty)$,
$\so(\infty)$, and $\sp(\infty)$ are pairwise equivalent. None of
them is equivalent to any non-finitary diagonal Lie algebra.

\item[b)] Let $\ss_1=X_1(\Cal T_1)$ and $\ss_2=X_2(\Cal T_2)$ be
infinite-dimensional non-finitary diagonal Lie algebras. Set
$S_i=\mathrm{Stz}(\Cal S_i)$, $S=\mathrm{GCD}(S_1,S_2)$,
$R_i=\quotst(S_i,S)$, $\delta_i=\delta(\Cal T_i)$,
$C_i=\mathrm{Stz}(\Cal C_i)$, $C=\mathrm{GCD}(C_1,C_2)$,
$B_i=\quotst(C_i,C)$, and $\sigma_i=\sigma(\Cal T_i)$ for $i=1,2$.
Then $\ss_1$ is equivalent to $\ss_2$ if and only if the following
conditions hold.
\begin{itemize}
\item[1)] $S_1\stackrel{\QQ}{\sim}S_2$.

\item[2)] Both $\ss_1$ and $\ss_2$ are either sparse or
non-sparse.

\item[3)] If $\ss_1$ and $\ss_2$ are non-sparse and $S$ is not
divisible by an infinite power of any prime number, then:
\begin{itemize}
\item[3.1)] $\frac{R_1}{\delta_1}=\frac{R_2}{\delta_2}$;

\item[3.2)] $\ss_1$ and $\ss_2$ have the same density type;

\item[3.3)] $\ss_1$ and $\ss_2$ are of the same type ($X_1=X_2$);

\item[3.4)] $\ss_1$ and $\ss_2$ have the same symmetry type;

\item[3.5)] $C_1\stackrel{\QQ}{\sim}C_2$ if $\ss_1$ and $\ss_2$
are two-sided non-symmetric;

\item[3.6)] $\frac{R_1\sigma_1}{B_1}=\frac{R_2\sigma_2}{B_2}$ if
$\ss_1$ and $\ss_2$ are two-sided strongly non-symmetric and $C$
is not divisible by an infinite power of any prime number.
\end{itemize}
\end{itemize}
\end{itemize}
\end{corollary}
\begin{proof}
a) The statement follows directly from \refth{Theorem} a).

b) To prove sufficiency it is easy to check case by case that all
the conditions of \refth{Theorem} b) are satisfied for both pairs
$\ss_1\subset\ss_2$ and $\ss_2\subset\ss_1$.

Let us prove necessity. Assume that there exist injective
homomorphisms $\ss_1\rightarrow\ss_2$ and $\ss_2\rightarrow\ss_1$.
Conditions 1 and 2 are obviously satisfied. Suppose that $\ss_1$
and $\ss_2$ are both non-sparse and $S$ is not divisible by an
infinite power of any prime number. Then
$\epsilon_1\frac{R_1}{\delta_1}\leq\frac{R_2}{\delta_2}$ and
$\epsilon_2\frac{R_2}{\delta_2}\leq\frac{R_1}{\delta_1}$ by
\refth{Theorem} b). Clearly, this is only possible if
$\epsilon_1=\epsilon_2=1$ and
$\frac{R_1}{\delta_1}=\frac{R_2}{\delta_2}$. Then $\ss_1$ and
$\ss_2$ have the same density type (otherwise one of the
inequalities would be strict). Conditions 3.3$-$3.6 follow from
conditions 3.1$-$3.4 of \refth{Theorem} b) for both pairs
$(\ss_1,\,\ss_2)$ and $(\ss_2,\,\ss_1)$.
\end{proof}

\textbf{Remark.} Isomorphic Lie algebras are clearly equivalent.
If two Lie algebras satisfy \refth{BZh1} (or \refth{BZh2}), then
they satisfy also \refcor{Corollary1}. One can check that
conditions $\Cal A_3$ and $\Cal B_3$ of \refth{BZh1} correspond
respectively to conditions 3.1 and 3.6 of \refcor{Corollary1}.

Let $\mathbb D$ denote the set of equivalence classes of
infinite-dimensional diagonal Lie algebras. If we write
$\ss_1\to\ss_2$ in case there exists an injective homomorphism
from $\ss_1$ into $\ss_2$, then the relation $\to$ is a partial
order on $\mathbb D$. It follows from \refth{Theorem} a) that
$\mathbb D$ has the only minimal element (which also is the least
element) with respect to the order $\to$: the equivalence class
consisting of the three finitary Lie algebras $\sl(\infty)$,
$\so(\infty)$, $\sp(\infty)$. The following statement shows that
there exist precisely one maximal element of $\mathbb D$ (which
also is the greatest element).

\begin{corollary}\label{Corollary2}
Let $\ss=X(\Cal T)$ be a diagonal Lie algebra. Set
$S=\mathrm{Stz}(\Cal S)$. The following are equivalent.

1) Any diagonal Lie algebra admits an injective homomorphism into
$\ss$.

2) $\ss$ is sparse and $S=p_1^\infty p_2^\infty\cdots$, where
$p_1,\,p_2,\dots$ is the increasing sequence of all prime numbers.

\end{corollary}
\begin{proof}
1)$\Rightarrow$2): Consider a Lie algebra $\ss'=A(\Cal T')$, where
$\Cal T'$ is sparse and $\displaystyle\mathrm{Stz}(\Cal
S')=p_1^\infty p_2^\infty\cdots$. Since $\ss'$ admits an injective
homomorphism into $\ss$, the Steinitz number $\quotst(p_1^\infty
p_2^\infty\cdots,S)$ is finite and $\ss$ is sparse by
\refth{Theorem} b). Then clearly $S=p_1^\infty p_2^\infty\cdots$.

2)$\Rightarrow$1): It follows immediately from \refth{Theorem}.
\end{proof}

The equivalence class corresponding to the maximal element of
$\mathbb D$ consists of infinitely many pairwise non-isomorphic
Lie algebras. Indeed, by \refth{BZh1} there is only one, up to
isomorphism, sparse one-sided Lie algebra of type $A$ satisfying
property 2 of \refcor{Corollary2}, but there are infinitely many
sparse two-sided Lie algebras of type $A$ with this property. By
\refth{BZh2}, any Lie algebra of type other than $A$ satisfying
property 2 of \refcor{Corollary2} is isomorphic to the sparse
two-sided symmetric Lie algebra of type $A$ with
$\mathrm{Stz}(\Cal S)=p_1^\infty p_2^\infty\cdots$.

\section{Acknowledgements}

I want to express my gratitude to my thesis adviser Ivan Penkov
for sharing his many ideas with me. This project would never have
been without his guidance and support. I also thank Elizabeth
Dan-Cohen and Ivan Penkov for their invaluable assistance in the
writing of this text.

\end{document}